\documentclass[12pt,a4paper,oneside,reqno]{amsart}

\usepackage[margin=2.5cm]{geometry}

\usepackage[utf8]{inputenc}
\usepackage{amsmath,amsfonts,bbm}
\usepackage{mathtools}
\usepackage{xspace}
\usepackage[english]{babel}
\usepackage[section]{placeins}
\usepackage{verbatim}
\usepackage{caption}
\usepackage{graphicx,color}
\usepackage[pdfusetitle]{hyperref}

\newtheorem{theorem}             {Theorem}[]

\newtheorem{definition} [theorem]{Definition}
\newtheorem{hypothesis} [theorem]{Hypothesis}

\newcommand{\R}{\mathbb{R}}

\newcommand{\T}{\textup{T}}     
\newcommand{\D}{\textup{D}}     
\renewcommand{\d}{{\textup{d}}} 
\newcommand{\Ddist}{\mathcal{D}}  

\newcommand{\der} [2]{\frac{\textrm{d} #1}{\textrm{d} #2}}
\newcommand{\pder}[2]{\frac{\partial #1}{\partial #2}}

\DeclarePairedDelimiter\norm{\lVert}{\rVert}

\DeclarePairedDelimiterX\inner[2]{\langle}{\rangle}{#1 , #2}
\DeclarePairedDelimiterX\poisson[2]{\{}{\}}{#1 , #2}
\DeclarePairedDelimiterX\lie[2]{[}{]}{#1 , #2}
\DeclarePairedDelimiterX\set[2]{\{}{\}}{#1 \mathrel{\delimsize|} #2}

\newcommand{\pr}{\textrm{pr}}

\newcommand{\NHIM}{\textsc{nhim}\xspace}

\newcommand{\ns}{\hphantom{-}}

\title[Realizing nonholonomic dynamics as limit of friction forces]%
{Realizing nonholonomic dynamics\\ as limit of friction forces}

\author{Jaap Eldering}
\email{jaap@jaapeldering.nl}
\address{%
  Universidade de São Paulo -- ICMC \\
  Avenida Trabalhador São-carlense 400 \\
  CEP 13566-590, São Carlos, SP, Brazil
}

\date{\today}

\begin{document}

\begin{abstract}
The classical question whether nonholonomic dynamics is realized as
limit of friction forces was first posed by Carathéodory. It is
known that, indeed, when friction forces are scaled to infinity, then
nonholonomic dynamics is obtained as a singular limit.

Our results are twofold. First, we formulate the problem in a
differential geometric context. Using modern geometric singular
perturbation theory in our proof, we then obtain a sharp statement on
the convergence of solutions on infinite time intervals.
Secondly, we set up an explicit scheme to approximate systems with
large friction by a perturbation of the nonholonomic dynamics. The
theory is illustrated in detail by studying analytically and
numerically the Chaplygin sleigh as example. This approximation scheme
offers a reduction in dimension and has potential use in applications.

MSC2010 numbers: 37J60, 70F40, 37D10, 70H09

Keywords: nonholonomic dynamics, friction, constraint realization,
singular perturbation theory, Lagrange mechanics
\end{abstract}

%

\maketitle

\section{Introduction}

Nonholonomic dynamics is a classical subject in mechanics that has
seen an increase of activity in the last decades with people studying
conserved quantities, symmetries and integrability, numerical
integrators, as well as various toy problems with intricate behavior
--- the rattleback and tippe top being the most famous ones --- and
has many applications in engineering sciences such as robotics, see
e.g.~\cite{Leon2012:hist-nonhol-mech, Bloch2003:nonholmech}.

We consider the fundamental question whether nonholonomic dynamics is
realized as the idealization, or limit, of large friction forces. Our
main results are twofold. First, we reprove previous results by
Brendelev~\cite{Brendelev1981:realnonhol} and
Karapetian~\cite{Karapetian1981:realnonhol} that answer the question
in the affirmative, but we treat the problem in a differential geometric setting,
see Theorem~\ref{thm:friction-limit}. Our use of geometric singular
perturbation theory provides an improved statement on the convergence
of solution curves on infinite time intervals. The examples in
Section~\ref{sec:diff-constraints} show that our result in general is
sharp. Secondly, we obtain an approximation scheme for the large
friction dynamics by an expansion in the singular perturbation
parameter. This allows modeling this non-ideal, large friction
dynamics on the reduced nonholonomic phase space with a modification
term added to the nonholonomic system. This has potential applications
for simplified modeling of systems where some slippage occurs, e.g.~in
explaining tippe top
inversion~\cite{Bou-Rabee2004:dissipation-tippetop} as well as in
engineering sciences, such as control of robots or submerged vehicles
where slipping/drift may become important, see
e.g.~\cite{Sidek2008:model-nonhol-robot-slip,
  Fedorov2010:hydro-chaplygin-sleigh}. As an example, we apply this
approximation scheme to the Chaplygin sleigh, see
Section~\ref{sec:sleigh-expand} and
Figures~\ref{fig:sleigh-allphaseplots}
and~\ref{fig:sleigh-alltrajplots}.

\subsection{Realizing constraints as idealizations}

Constraints in mechanical dynamical systems should often be seen as
simplifying idealizations of more intricate underlying models.
Examples are ubiquitous. A pendulum can be modeled as a point mass
moving on a circle --- or a sphere in 3D --- under gravity, but in physical
reality, the constraint of staying on the circle is only approximately
realized by a very stiff rod that connects the mass to the point of
suspension. A rigid body can be considered an approximation of a large
number of atoms, or as a continuous medium with a strong interaction
potential that keeps the body rigid. Finally, a smooth, convex body
rolling on a surface without slipping, or a figure skater sliding
without sideways movement are examples of the idealization of
friction forces, and modeled by nonholonomic mechanics.

The examples above show how both holonomically and nonholonomically
constrained systems can be considered idealizations of larger,
unconstrained systems under strong potentials or friction forces.
On the other hand, the dynamics for constrained systems is
conventionally postulated to be given by Lagrange's variational
principle (which is equivalent to Newton's laws in the unconstrained
case) together with d'Alembert's principle that reaction forces do no
work along virtual displacements that satisfy the constraints.
Thus, a natural question is whether these two different viewpoints
yield corresponding and experimentally correct equations of motion
for a given system.

In this paper we revisit the question whether friction forces can
realize nonholonomically\footnote{%
  The realization of holonomic constraints by strong potentials has
  been studied in~\cite{Rubin1957:motion-constr,
    Takens1980:motion-constr, Kozlov1990:real-hol-constr} and other
  papers. In that case the limit is somewhat subtle, and depends on
  the energy in oscillatory modes normal to the constrained manifold.
} constrained systems. That is, when we consider a mechanical system
with friction forces acting in only some of the directions and then
scale this friction to infinity, do we recover nonholonomic dynamics
in an appropriate limiting sense? This question dates back at least
to Carathéodory, who considered it for the Chaplygin
sleigh~\cite{Caratheodory1933:schlitten}, see
Section~\ref{sec:history} below for a brief history.

A second, related question is whether indeed the Lagrange--d'Alembert
principle for nonholonomic dynamics is `correct'. The relevance of
this question is indicated by the confusion in the late 19th century
about the correct formulation of nonholonomic dynamics,
see~\cite[p.~46--47]{Cushman2010:geom-nonhol-systems}
and~\cite[Sec.~3.2.3]{Marle1998:approaches-nonhol} for a discussion.
An alternative method to obtain equations of motion for
nonholonomically constrained systems, called `vakonomic mechanics',
was proposed by Kozlov~\cite{Kozlov1983:real-nonint-constr}. This
method produces different dynamics\footnote{%
  Vershik and Gershkovich showed that in a generic system with
  nonholonomic constraints the vakonomic and Lagrange--d'Alembert
  solutions are
  incompatible~\cite[Sec.~4.3]{Vershik1988:nonhol-theory-distrib}.%
}, and hence raises the question which method is `correct'. The quotes are
used since such questions ultimately have to be decided by physical
experiment, and the answer may depend on the kind of system studied.
For rolling mechanical systems the general consensus seems to be that
the d'Alembert principle is correct. Indeed, Lewis and Murray verified
experimentally that for a ball rolling on a turntable, nonholonomic
dynamics gives a better description than vakonomic
mechanics~\cite{Lewis1995:var-constrsys-th-exp}. On the other hand,
vakonomic dynamics, satisfying a variational principle, is applicable
in control theory and the motion of rigid bodies in
fluids~\cite{Kozlov1982:dyn-nonint-constr1,Kozlov1982:dyn-nonint-constr2}
as cited in~\cite{Fedorov2010:hydro-chaplygin-sleigh}. Alternative
methods to realize nonholonomic dynamics besides friction have been
proposed too. Bloch and Rojo~\cite{Bloch2008:quant-nonhol} (see
also~\cite{Bloch2010:nonhol-dissip-quant}) used coupling to an
external field to obtain nonholonomic dynamics\footnote{%
  Although the nonholonomic system appears as a slow manifold of this
  infinite-dimensional coupled system, the third line
  in~\cite[Eq.~(9)]{Bloch2008:quant-nonhol} seems to indicate that the
  fast dynamics of $\theta$ around $\alpha(x,t)$ is oscillatory,
  i.e.\ the normal directions are elliptic. Thus, convergence of the
  limit dynamics is not clear to me.%
}, while Ruina~\cite{Ruina1998:nonhol-stab-piecewise} shows that
mechanical systems with intermittent contact, i.e.\ with piecewise
holonomic constraints, can be viewed as nonholonomic systems in the
limit as the contact switches more frequently.

Although these `correctness' questions should be decided by
experiment, the idea of trying to realize constraint dynamics from the
unconstrained system using an underlying principle can offer
theoretical insights. For example, whereas nonholonomic dynamics based
on d'Alembert's principle can be viewed as a limit of friction forces,
vakonomic mechanics can be viewed as letting the mass in the
mechanical metric go to infinity along constrained
directions~\cite{Kozlov1983:real-nonint-constr}. See
also~\cite{Kozlov1992:realconstr} for a discussion of various methods
to realize constraints in dynamics
and~\cite[Sec.~0.3]{Borisov2015:exp-motion-sliding} for a review of
the applicability of nonholonomic dynamics to sliding and rolling
bodies.

Showing that a fundamental principle such as friction --- or large
mass terms in the vakonomic principle --- can realize the constrained
system provides a fundamental justification for the constrained
equations of motion as an idealized limit. But conversely, given
constrained equations of motion, it also allows one to study under
precisely which conditions the underlying principle leads to this
realization. For example, in the case of friction realizing nonholonomic
dynamics: instead of linear (viscous) friction forces, do Coulomb-like
forces also realize nonholonomic dynamics?
In~\cite{Kozlov2010:dry-fric-nonhol} this is considered for the
specific case of a homogeneous rolling ball. Finally, when the limit is
sufficiently well-behaved, e.g.\ in the case of linear friction, one can
moreover calculate correction terms to the idealized dynamics. For the
friction limit, these represent the effect of large, but finite
friction forces, see Section~\ref{sec:beyond-limit}.

\subsection{A brief historic overview}
\label{sec:history}

To our knowledge, Carathéodory~\cite{Caratheodory1933:schlitten} first
explicitly addressed the question whether friction forces can realize
nonholonomic dynamics, although his starting sentence that
``nonholonomic motion in mechanics is known to be caused by friction
forces'' indicates that the idea is older. His study of the Chaplygin
sleigh (without attributing it to
Chaplygin~\cite{Chaplygin1911:motion-nonhol}) includes a complete
solution in quadratures and finally addresses the question whether the dynamics can
be realized by a viscous friction force. Carathéodory's conclusion is
negative, but this is due to an error in his reasoning as
discussed thoroughly in~\cite{Fufaev1964:real-nonhol-visc-fric} (see
also~\cite[p.~233--237]{Neimark1972:dynnonhol}).
That is, Carathéodory fixes initial conditions for the
velocity perpendicular to the skate blade, but these are not on the
slow manifold associated to the nonholonomic dynamics. See also the
exposition in Section~\ref{sec:sleigh}, but note that compared
to~\cite{Caratheodory1933:schlitten,Fufaev1964:real-nonhol-visc-fric} we retain $v$,
the velocity perpendicular to the sleigh's skate blade, instead of
rewriting it into a second-order equation for $\omega$, the angular
velocity of the sleigh.

The realization of nonholonomic constraints by friction forces was
independently shown by Brendelev~\cite{Brendelev1981:realnonhol} and
Karapetian~\cite{Karapetian1981:realnonhol}. They proved that indeed
an unconstrained system with viscous friction forces added to it will
in the limit converge to the nonholonomic system whose constraint
distribution is defined by the friction forces being zero on it. The
limit is understood to be in the sense that solution curves of the
unconstrained system converge to solutions of the nonholonomically
constrained system, uniformly on each bounded time interval $[t_1,T]$
with $0 < t_1 < T$, as the friction is scaled to infinity. Both
authors used Tikhonov's theorem~\cite{Tikhonov1952:sys-DE-small-param}
for singularly perturbed systems to obtain this result in a local
coordinate setting. Kozlov combined this friction limit into a larger
framework including both friction and inertial terms, obtaining both
nonholonomic and vakonomic dynamics as limits in special
cases~\cite{Kozlov1992:realconstr}.

\subsection{Outline and contributions of the paper}

After the introduction, we first briefly review nonholonomic dynamics
according to the Lagrange--d'Alembert principle.
In Section~\ref{sec:diff-constraints} we show how different approaches
using potentials or friction to realize the same constrained
system yield qualitatively different solutions, arbitrarily close to
their respective limits. Then we start with the main part of the paper
and present how friction forces realize nonholonomic dynamics. First
with the Chaplygin sleigh as example. In Section~\ref{sec:sleigh} we
derive its equations of motion as a nonholonomic system and in
Section~\ref{sec:sleigh-friction} as an unconstrained system, but with
friction added. Secondly we treat
the general theory in Section~\ref{sec:general}. We show how the
Lagrange--d'Alembert equations are recovered as a limit, using modern
geometric singular perturbation theory. In keeping the exposition
clear, we shall not try to attain the utmost generality.
That is, we restrict to Lagrangian systems of mechanical type and
consider only time-independent, linear friction forces. This means
that we only recover linear nonholonomic constraints\footnote{%
  Affine constraints do appear naturally, for example, for a ball
  rolling on a turntable. However, fully nonlinear constraints in
  mechanical systems seem scarce: according
  to~\cite[p.~213,223--233]{Neimark1972:dynnonhol} all known examples
  are variations of the one due to
  Appell~\cite{Appell1911:ex-mouv-rel-nonlin-vitesse}. See
  also~\cite{Marle1998:approaches-nonhol} for a discussion of
  nonholonomic constraint principles, including nonlinear
  constraints.%
}. Instead, we shall present various simple examples to illustrate
some of the issues discussed above.

The first aim of this paper is to provide an accessible
exposition of the problem of realizing nonholonomically constrained
systems as a limit of infinite friction forces, using the Chaplygin
sleigh as a canonical example. Secondly, we
present the problem in a more geometric setting
than~\cite{Brendelev1981:realnonhol,Karapetian1981:realnonhol}. That
is, we express the problem in a differential geometric formulation and
use geometric singular perturbation theory, as developed by
Fenichel~\cite{Fenichel1979:singpertODE} and  based on the theory of
normally hyperbolic invariant manifolds (abbreviated as \NHIM), see
e.g.~\cite{Fenichel1971:invarmflds, Hirsch1977:invarmflds,
  Eldering2013:NHIM-noncompact}. This allows us to phrase the results
independently of coordinate charts and provide improvements on
the results obtained in~\cite{Brendelev1981:realnonhol,
  Karapetian1981:realnonhol}; in particular, we can describe more
precisely in what sense the solution curves converge on positively
unbounded time intervals. Moreover, persistence of the \NHIM allows us
to rigorously study a perturbation expansion of the nonholonomic
system away from the infinite friction limit and obtain corrections to
the nonholonomic dynamics as an expansion for small $\varepsilon > 0$,
where $1/\varepsilon$ is the scale-factor multiplying the friction
forces. This can be used to describe systems with large, but finite
friction forces as nonholonomic systems with small correction terms.
We derive general formulas for this expansion and illustrate their
use with an analytical and numerical study of the
approximation of the Chaplygin sleigh with large friction in
Section~\ref{sec:sleigh-expand}. Efforts towards obtaining such a
general expansion were made in~\cite{Wang1996:creep-dyn-nonhol}, which
studies `creep dynamics' for a few example systems. For the Chaplygin
sleigh they find the same first order correction term (our $h^{(1)}$)
to the invariant manifold $\Ddist_\varepsilon$,
but their modified equations of motion are not specified.

\section{Nonholonomic dynamics}
\label{sec:NH-dynamics}

Let us here briefly recall the Lagrange--d'Alembert formalism for
nonholonomic dynamics and establish notation. A Lagrangian system is
given by a Lagrangian function $L\colon \T Q \to \R$, where $Q$ is an
$n$-dimensional smooth manifold. We shall assume that the Lagrangian
is of mechanical type, i.e.\ of the form
\begin{equation}\label{eq:Lagrangian}
  L(q,\dot{q}) = \frac{1}{2}\kappa_q(\dot{q},\dot{q}) - V(q),
  \qquad (q,\dot{q}) \in \T Q,
\end{equation}
where $\frac{1}{2}\kappa_q(\dot{q},\dot{q})$ is the kinetic energy,
which is given by a Riemannian metric $\kappa$ on $Q$, and
$V\colon Q \to \R$ is the potential energy. This implies that $L$ is
hyperregular and that the Euler--Lagrange equations of motion have
unique solutions. Let $q$ be local coordinates on $Q$ and let
$(q,\dot{q})$ denote induced coordinates on $\T Q$. Then the
Euler--Lagrange equations are given by\footnote{%
  Note that $[L] := [L]_i\,\d q^i$ can intrinsically be viewed as a
  map from $J_0^2(\R;Q)$, the space of second-order jets of curves in
  $Q$, to $\T^* Q$, covering the identity map on $Q$.%
}
\begin{equation}\label{eq:EL-equations}
  [L]_i := \der{}{t} \pder{L}{\dot{q}^i} - \pder{L}{q^i} = 0.
\end{equation}

A linear nonholonomic constraint is imposed on this system by
specifying a regular distribution, i.e.\ a smooth, constant rank $k$
vector subbundle, $\Ddist \subset \T Q$. Kinematically, this
constraint specifies the allowed velocities of the system, that is, a
velocity vector $\dot{q} \in T_q Q$ is allowed precisely if
$\dot{q} \in \Ddist_q$.

To specify the nonholonomic dynamics with this constraint, an extra
principle has to be imposed, as solutions of~\eqref{eq:EL-equations}
would generally violate the kinematic constraint imposed by $\Ddist$.
To guarantee that the solutions satisfy the constraint, we add a
constraint reaction force $F_c$ to the right-hand side
of~\eqref{eq:EL-equations} according to d'Alembert's principle, see
e.g.~\cite[p.~4--6]{Cushman2010:geom-nonhol-systems}.
\begin{hypothesis}[d'Alembert's principle]\label{hyp:dAlembert}
  The constraint forces $F_c$ do no work along movements that are
  compatible with the constraints. That is, for any
  $(q,\dot{q}) \in \Ddist$ we have that
  $F_c(q,\dot{q}) \cdot \dot{q} = 0$.
\end{hypothesis}
From now on we shall implicitly assume this principle for
nonholonomically constrained systems. This leads to the
Lagrange--d'Alembert equations, which again have unique solutions.
Abstractly, these state that a smooth curve $c\colon (t_0,t_1) \to Q$
is a solution of the equations if for all $t \in (t_0,t_1)$
\begin{equation}\label{eq:LdA-equations}
  [L]\bigl(\ddot{c}(t)\bigr) \in \Ddist^0
  \qquad\text{and}\qquad
  \dot{c}(t) \in \Ddist.
\end{equation}
Here $\Ddist^0 \subset \T^* Q$ denotes the annihilator of $\Ddist$,
that is, all covectors $\alpha \in \T^* Q$ such that $\alpha(\dot{q}) = 0$
for any $\dot{q} \in \Ddist_{\pi(\alpha)}$.

Let $\Ddist$ locally be given by a set of $n-k$ independent constraint
one-forms $\zeta^a = C^a_i(q) \, \d q^i$,
i.e.~$\Ddist = \set{(q,\dot{q}) \in \T Q}{\forall a \in [1,n\!-\!k]\colon \zeta^a_q(\dot{q})=0}$.
Then the Lagrange--d'Alembert equations in local coordinates are of
the form
\begin{equation}\label{eq:LdA-coords}
  \der{}{t}\pder{L}{\dot{q}^i} - \pder{L}{q^i} = \lambda_a\,C^a_i(q)
  \qquad\text{and}\qquad
  \zeta^a_q(\dot{q}) = 0,
\end{equation}
where $\lambda_a$ are Lagrange multipliers that have to be solved for,
and the constraint force is given by
$F_c = \lambda_a\,\zeta^a$.

\section{An example of different constraint realizations}
\label{sec:diff-constraints}

The following toy example illustrates how different choices for
realizing the same constrained system lead to qualitatively different
dynamics close to the limit. These qualitative differences are
essentially due to the fact that the limits of sending time to
infinity or a singular perturbation parameter to zero do not
commute. A nice illustration of this simple fact is given by
Arnold~\cite[p.~65]{Arnold1978:ODEs}. Consider a bucket of water with
a hole of size $\varepsilon$ in the bottom. After any fixed, finite
time $t$ the water level remains unchanged as $\varepsilon \to 0$, but
for any $\varepsilon > 0$ the bucket becomes empty as $t \to \infty$.

We consider the mathematical pendulum in the vertical plane. We view
it as a system constrained to the circle in a holonomic, nonholonomic
and vakonomic sense. That is, we view the constraint as generated by a
strong potential, strong friction forces and large inertia, and
compare the dynamics when the respective constraining terms
are large but finite.

The unconstrained system is given by the Lagrangian
\begin{equation}\label{eq:pendulum-Lagr}
  L(x,y,\dot{x},\dot{y}) = \frac{1}{2}(\dot{x}^2 + \dot{y}^2) - g\,y
\end{equation}
on $\T \R^2$, where $g$ is the gravitational acceleration and a unit
mass has been chosen.

Viewing the pendulum as a mass attached to the origin by a perfectly
stiff rod, we can choose as constraining potential
\begin{equation}\label{eq:pendulum-constr-pot}
  V_\varepsilon(x,y) = \frac{1}{2\varepsilon}(\sqrt{x^2 + y^2} - 1)^2,
\end{equation}
that is, the square distance from the constrained submanifold
$S^1 \subset \R^2$, modeling a perfectly stiff rod as
$\varepsilon \to 0$. We can now easily apply the results
from~\cite{Takens1980:motion-constr} as $S^1$ is a codimension one
manifold and the constraining potential has constant second derivative
in the normal direction. It follows that the limit motion is precisely
given by the standard pendulum on $S^1$. Energy conservation also
makes it immediately clear that for small $\varepsilon$, solutions
stay close to the constraint manifold $S^1$. See the top image in
Figure~\ref{fig:pendulum}.

Secondly, the constrained system can be obtained from a
`nonholonomic' limit of adding friction forces in the radial
direction. This can be thought of as a model for a leaf falling or
`swirling down' under gravity where the leaf surface is tangential to
circles around the origin and there is a large air resistance against
perpendicular movement. We model this air resistance by the friction
force defined by the Rayleigh function
\begin{equation}\label{eq:pendulum-constr-fric}
  R_\varepsilon(x,y,\dot{x},\dot{y})
  = \frac{1}{2\varepsilon}\Bigl(\frac{x \dot{x} + y \dot{y}}{\sqrt{x^2 + y^2}}\Bigr)^2,
\end{equation}
or $R_\varepsilon = \dot{r}^2/(2\varepsilon)$ in polar coordinates. The
friction force is given by $F = -\frac{\partial R}{\partial \dot{q}}$. In
the limit of $\varepsilon \to 0$ this gives rise to a nonholonomic
system whose constraint distribution is actually integrable; the
associated foliation consists of concentric circles. Although the
limit dynamics is the same as with the constraining potential, the
dynamics for any finite $\varepsilon > 0$ is qualitatively different:
the leaf will not stay close to the original submanifold $S^1$, but slowly
drift down. This can most easily be seen for a leaf with initial
conditions $x = \dot{x} = 0$. Under the combined forces of gravity and
friction, it will settle to a vertical speed of
$\dot{y} = -\varepsilon\,g$. Hence, over long times the constraint
manifold $S^1$ is not even approximately preserved. See the middle
image in Figure~\ref{fig:pendulum}.

Finally, we consider the system as a limit of adding a large inertia
term in the radial direction. That is, we add to the Lagrangian a term
\begin{equation}\label{eq:pendulum-constr-vako}
  I = \frac{1}{2\varepsilon}\Bigl(\frac{x \dot{x} + y \dot{y}}{\sqrt{x^2 + y^2}}\Bigr)^2,
\end{equation}
or in polar coordinates $I = \dot{r}^2/(2\varepsilon)$. In the limit
of $\varepsilon \to 0$ this realizes vakonomic dynamics,
see~\cite{Kozlov1983:real-nonint-constr}. Since the constraint
distribution is integrable, the resulting limit dynamics is equal to
the nonholonomic limit with friction and the holonomic limit with
potential forces on $S^1$. The dynamics for finite $\varepsilon$
resembles most closely that of the friction model, but there is a
difference that can most clearly be observed by considering again the
initial conditions $x = \dot{x} = 0$. In this case the equation for
$y$ takes the form
\begin{equation*}
  (1+\frac{1}{\varepsilon})\ddot{y} = -g.
\end{equation*}
Thus, instead of settling on a slow descend, the particle accelerates
downward, but with a slowed acceleration due to the added inertial
term. See the bottom image in Figure~\ref{fig:pendulum}.

The conclusion is that even though for all three constraint
realization methods, solutions converge on a fixed time interval to
the solutions of the constraint system --- which \emph{in this case}
is identical for the three methods --- this does not hold true on
unbounded time intervals. Indeed, for these three different methods,
we see three different behaviors when $t \to \infty$ for fixed
$\varepsilon > 0$. With the constraining potential the system stays
uniformly\footnote{%
  However, solutions do not uniformly in time converge to the
  constrained solutions. Consider as a counterexample the pendulum
  without gravity and an initial velocity along $S^1$: in the case of
  a finite potential, the particle will oscillate slightly away from
  $S^1$, hence trade some of its kinetic for potential energy and thus
  have on average a lower angular velocity than the perfectly
  constrained particle. This builds up over time to arbitrary large
  errors in position.%
} close to the initial leaf $S^1$ of the constraint distribution
$\Ddist = \{ \dot{r} = 0 \}$. With the friction force, this is not true
anymore, but the solution does stay close to $\Ddist$ as a submanifold
of $\T\R^2$. This is intuitively clear in the current example as
friction prevents $\dot{r}$ from growing large, and can be proven in
general under reasonable conditions, see Section~\ref{sec:general}.
Finally, with the inertial term, the pendulum example shows that
solutions can even diverge arbitrarily far from $\Ddist$.

\begin{figure}[bp]
  \centering
  \includegraphics[width=10cm]{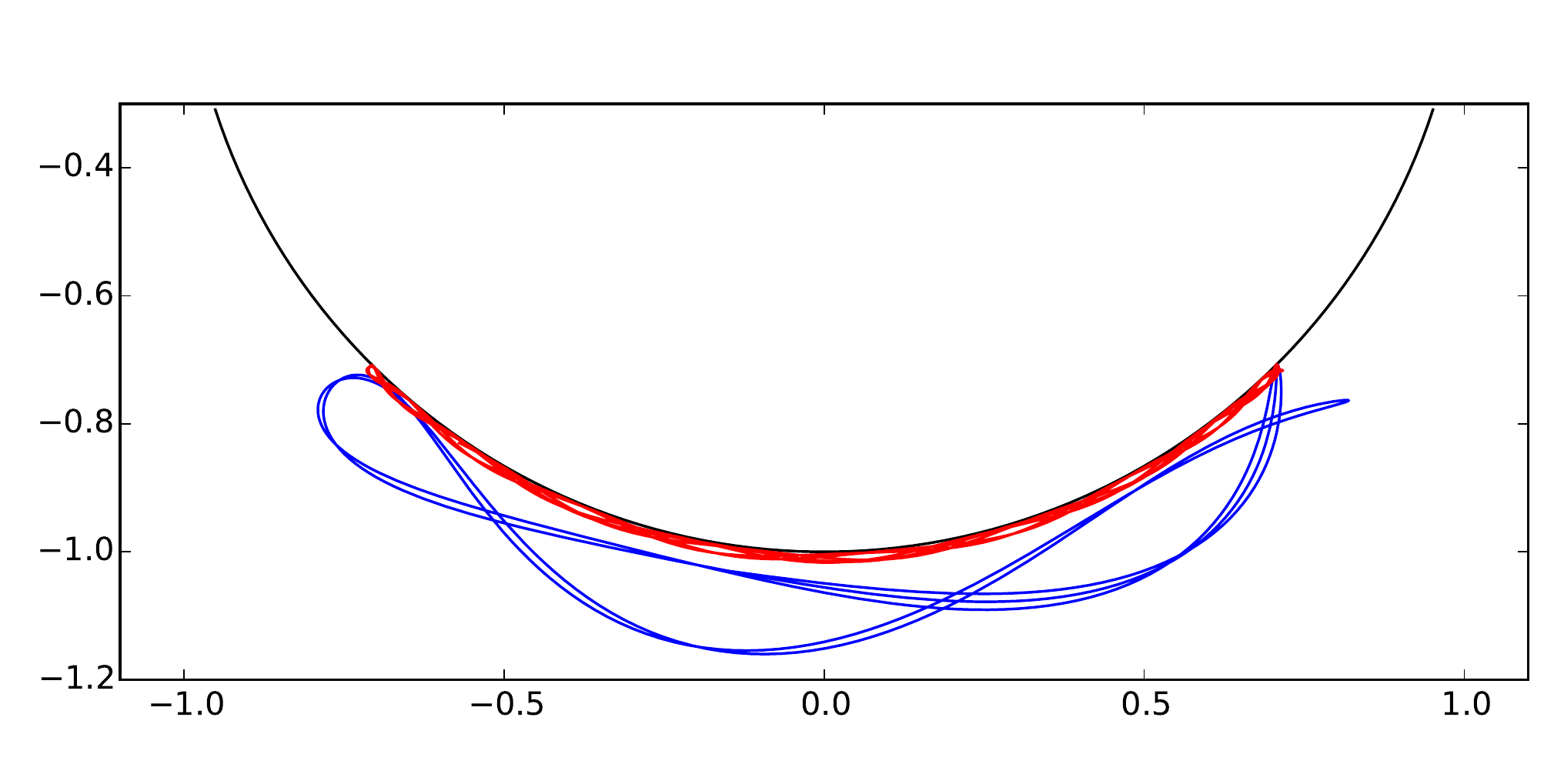}\\
  \includegraphics[width=10cm]{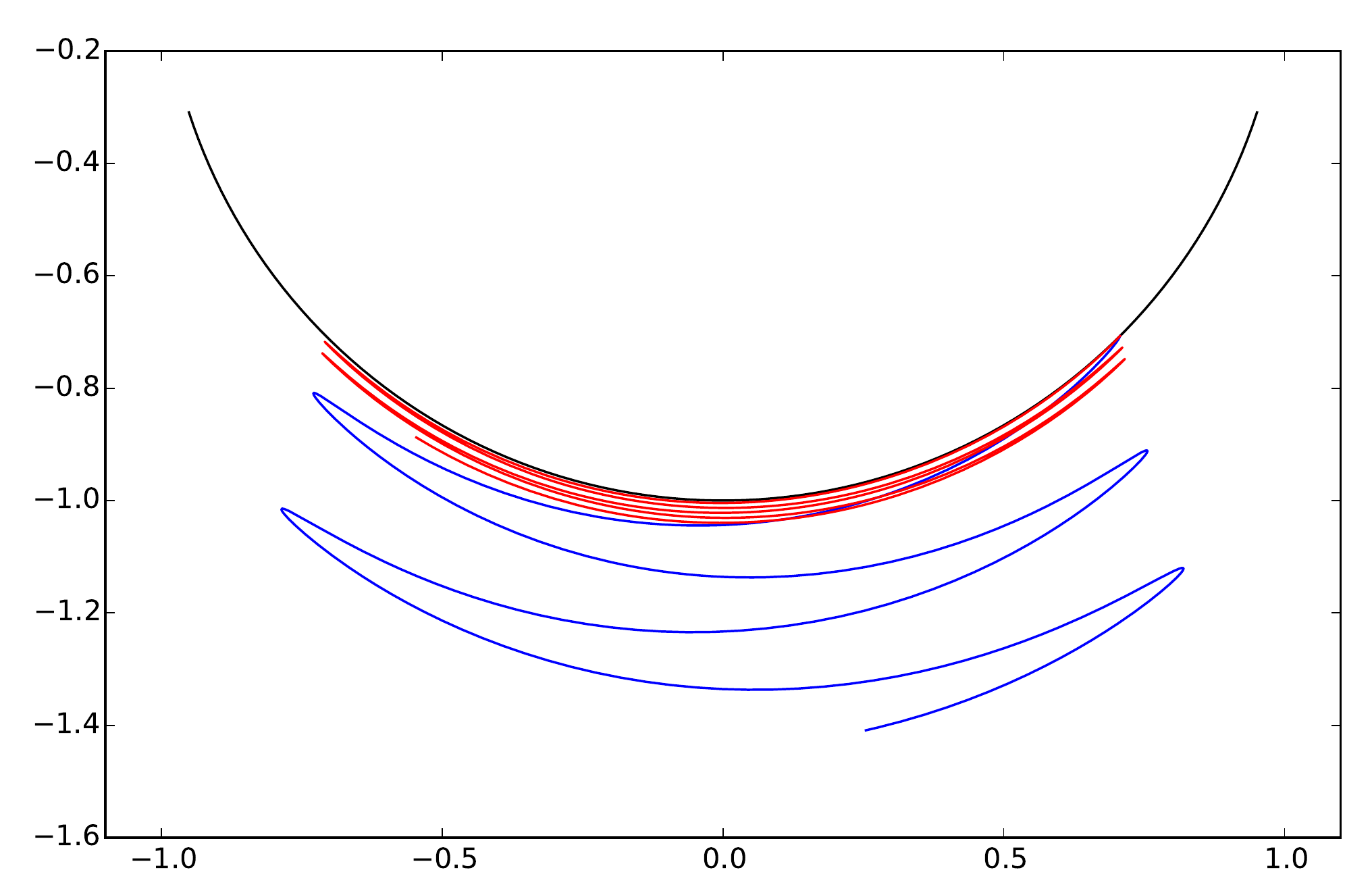}\\[10pt]
  \hspace{0.2cm}\includegraphics[width=10cm]{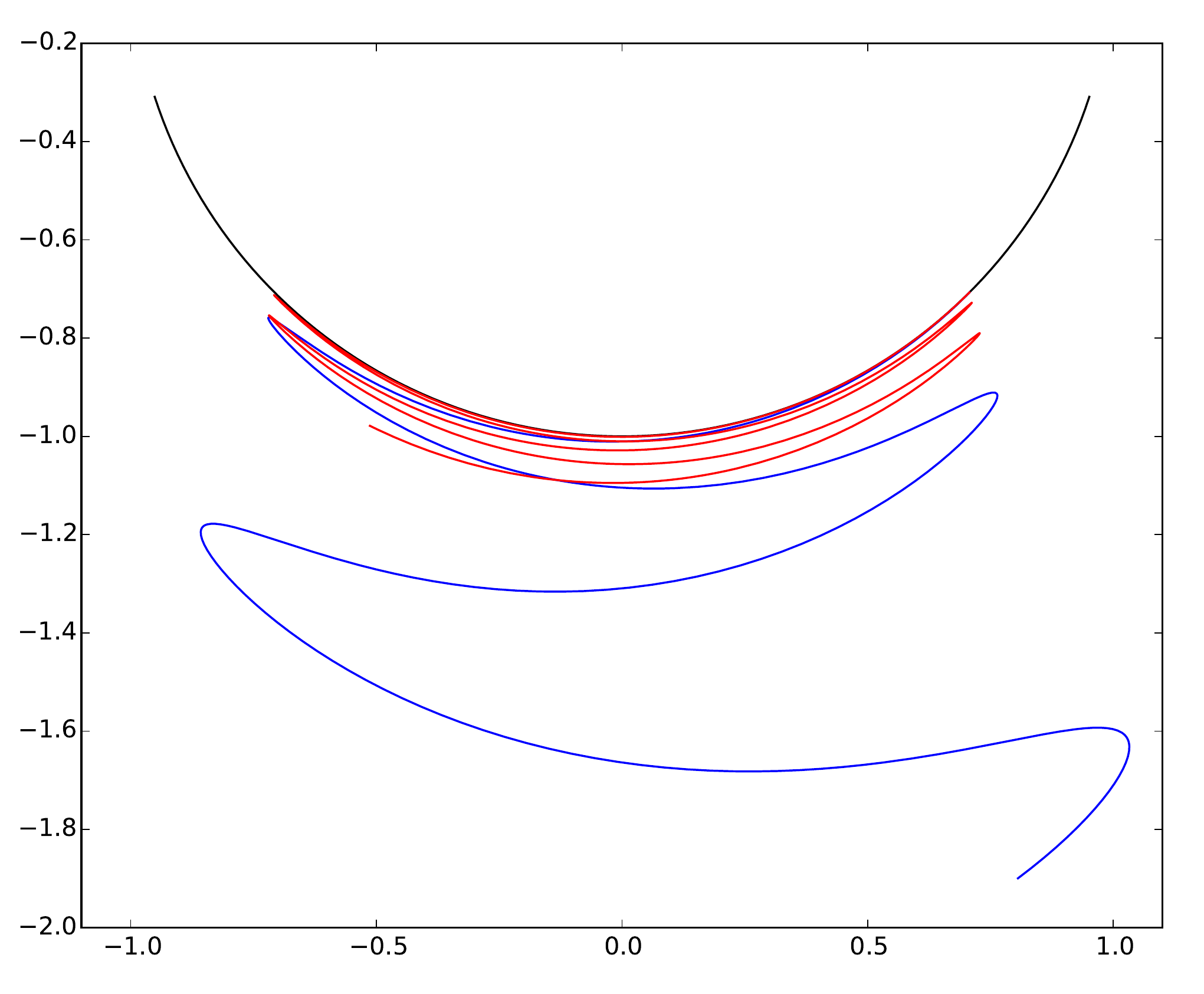}\\
  \caption{Top to bottom: numerical simulations of  potential, frictional
    and inertial constraint approximations. The particle initially
    starts on $S^1$ at $45^\circ$ angle and with zero velocity. Its
    trajectory is shown for $10$ seconds with $g = 9.81$ and values
    $\varepsilon = 0.01$ in blue and $\varepsilon = 0.001$ in red.}
  \label{fig:pendulum}
\end{figure}

\section{The Chaplygin sleigh}
\label{sec:sleigh}

The Chaplygin sleigh~\cite{Chaplygin1911:motion-nonhol}
(see~\cite{Chaplygin1911:motion-nonhol-english2008} for a translation
into English) is a simple, yet interesting example of a mechanical
nonholonomic system. The sleigh is a body that can move on the plane,
but one of its contact points is a skate, see Figure~\ref{fig:sleigh}.
Alternatively, one can think of the contact point as a wheel that is
fixed to the body. The contact point can only move in the direction
along the skate blade/wheel, not in the perpendicular direction. The
other two ground contacts can move freely without constraint. The
center of mass is located a distance $a$ away from the skate contact
point along the line of the blade\footnote{%
  An offset of the center of mass perpendicular to the skate blade
  does not intrinsically change the system, since a shift of the
  contact point perpendicular to the skate blade does not alter the
  constraint.%
}.

\begin{figure}[htb]
  \centering
  \parbox[b]{11cm}{%
    \centering
    \input{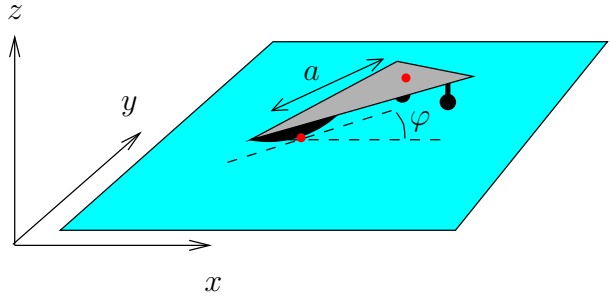}
    \caption{A Chaplygin sleigh: the left red dot is the point of
      contact of the skate, the right red dot is the center of mass.}
    \label{fig:sleigh}
  }
\end{figure}

An interesting feature of this nonholonomic system is that it exhibits
center-stable and center-unstable relative equilibria, something that
cannot occur in purely symplectic\footnote{%
  This does not hold anymore on Poisson manifolds; the Chaplygin
  sleigh can actually be realized with an adapted Poisson bracket
  $\poisson{u}{\omega} = \frac{a \omega}{I + m a^2}$ with respect to
  the moving frame coordinates introduced below.%
} Hamiltonian systems, where eigenvalues must always occur in pairs or
quadruples symmetric about the real and imaginary axes.

We shall explicitly recover the equations of motion for this system in
the Lagrangian setting\footnote{%
  See~\cite[Chap.~5]{Cushman2010:geom-nonhol-systems} for a thorough
  account how to derive the equations in various different settings,
  and the end of the chapter for some interesting notes.%
} and then illustrate in detail its realization through friction.

We describe the system with coordinates $(x,y) \in \R^2$ for the skate
contact point and $\varphi \in S^1$ for the angle of the skate blade
with the $x$ axis. Assume that the sleigh has total mass $m$ and
moment of inertia $I$ about the center of mass. To construct the
Lagrangian, we first express the center of mass point as
\begin{equation*}
  (x_c,y_c) = (x,y) + a\bigl(\cos(\varphi),\sin(\varphi)\bigr).
\end{equation*}
Then the Lagrangian is given by
\begin{equation}\label{eq:Lagr-sleigh}
  \begin{split}
    L &= \frac{1}{2}m\bigl(\dot{x}_c^2 + \dot{y}_c^2\bigr)
        +\frac{1}{2}I\dot{\varphi}^2 \\[5pt]
      &= \frac{1}{2}m\bigl(\dot{x}^2 + \dot{y}^2\bigr)
        +\frac{1}{2}\bigl(I + ma^2\bigr)\dot{\varphi}^2
        +ma\dot{\varphi}\bigl(-\dot{x}\sin(\varphi)+\dot{y}\cos(\varphi)\bigr).
  \end{split}
\end{equation}

We now switch to moving frame coordinates $(u,v,\omega)$, where $u$
and $v$ are the velocities parallel and perpendicular to the skate,
respectively, and $\omega$ is the angular velocity. They are given by
\begin{equation}\label{eq:sleigh-moving-frame}
  \begin{split}
    u &= \ns \dot{x}\cos(\varphi) + \dot{y}\sin(\varphi), \\
    v &=   - \dot{x}\sin(\varphi) + \dot{y}\cos(\varphi), \\
    \omega &= \dot{\varphi}.
  \end{split}
\end{equation}
This moving frame is aligned with the constraint
distribution $\Ddist$ in the sense that the constraint is
described by the function
\begin{equation*}
  \zeta_q(\dot{q}) := v = -\dot{x}\sin(\varphi) + \dot{y}\cos(\varphi),
\end{equation*}
i.e.\ the constraint $v = 0$ precisely expresses that the skate cannot
move sideways.

It is now a straightforward exercise to calculate the
Lagrange--d'Alembert equations~\eqref{eq:LdA-coords}. We
find for each of the coordinates $x,y,\varphi$ and their associated
velocities:
\begin{equation}\label{eq:sleigh-LdA}
  \begin{split}
    \der{}{t}\pder{L}{\dot{x}} - \pder{L}{x}
    &= m\ddot{x} + ma\bigl(  -\ddot{\varphi}\sin(\varphi) -\dot{\varphi}^2\cos(\varphi)\bigr)
     =  -\lambda \sin(\varphi), \\[5pt]
    \der{}{t}\pder{L}{\dot{y}} - \pder{L}{y}
    &= m\ddot{y} + ma\bigl(\ns\ddot{\varphi}\cos(\varphi) -\dot{\varphi}^2\sin(\varphi)\bigr)
     =\ns\lambda \cos(\varphi), \\[5pt]
    \der{}{t}\pder{L}{\dot{\varphi}} - \pder{L}{\varphi}
    &= (I + ma^2)\ddot{\varphi} \!
       \begin{aligned}[t]
       &+ma \der{}{t}   \Bigl[-\dot{x}\sin(\varphi)+\dot{y}\cos(\varphi)\Bigr] \\[3pt]
       &+ma\dot{\varphi}\bigl( \dot{x}\cos(\varphi)+\dot{y}\sin(\varphi)\bigr) = 0.
       \end{aligned}
  \end{split}
\end{equation}
Next, we switch to moving frame coordinates by substituting
in~\eqref{eq:sleigh-moving-frame} and taking linear combinations of
the first two equations with factors $\cos(\varphi)$ and
$\sin(\varphi)$. This yields
\begin{equation}\label{eq:sleigh-LdA-frame}
  \begin{split}
    m(\dot{u} - v\omega - a \omega^2)                &= 0, \\
    m(\dot{v} + u\omega + a\dot{\omega})             &= \lambda, \\
    (I + ma^2)\dot{\omega} + m a(\dot{v} + u \omega) &= 0.
  \end{split}
\end{equation}
Finally, we use the constraint $v = 0$, which implies that $\dot{v} = 0$.
Inserting this, we obtain the usual Chaplygin sleigh equations
\begin{equation}\label{eq:sleigh-equations}
  \dot{u} = a \omega^2, \qquad
  \dot{\omega} = - \frac{m a u \omega}{I + ma^2},
\end{equation}
with Lagrange multiplier $\lambda = m(u\omega + a\dot{\omega})$ giving
rise to the constraint force
\begin{equation}\label{eq:sleigh-constr-force}
  F_c = m(u\omega + a\dot{\omega})
        \begin{pmatrix} -\sin(\varphi) \\ \ns\cos(\varphi) \end{pmatrix}.
\end{equation}

\begin{figure}[htb]
  \captionsetup{margin=0.6cm}
  \centering
  \begin{minipage}[t]{7.7cm}
    \centering
    \includegraphics[width=7.7cm]{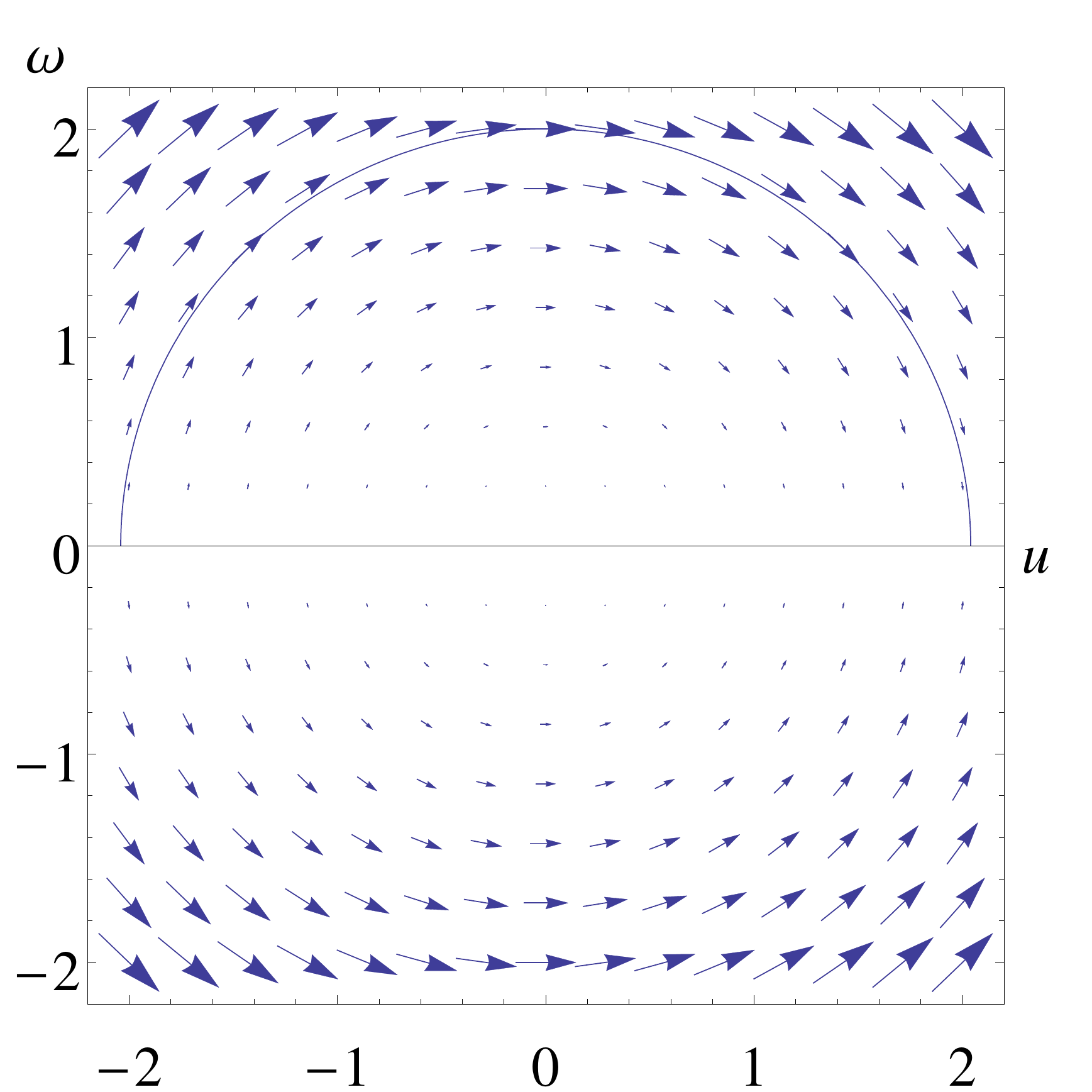}
    \caption{Phase plot of the $(u,\omega)$ coordinates with
      $m = I = 1$ and $a = 1/5$.}
    \label{fig:sleigh-phaseplot}
  \end{minipage}
  \hspace{0.2cm}
  \begin{minipage}[t]{7.7cm}
    \centering
    \includegraphics[width=7.7cm]{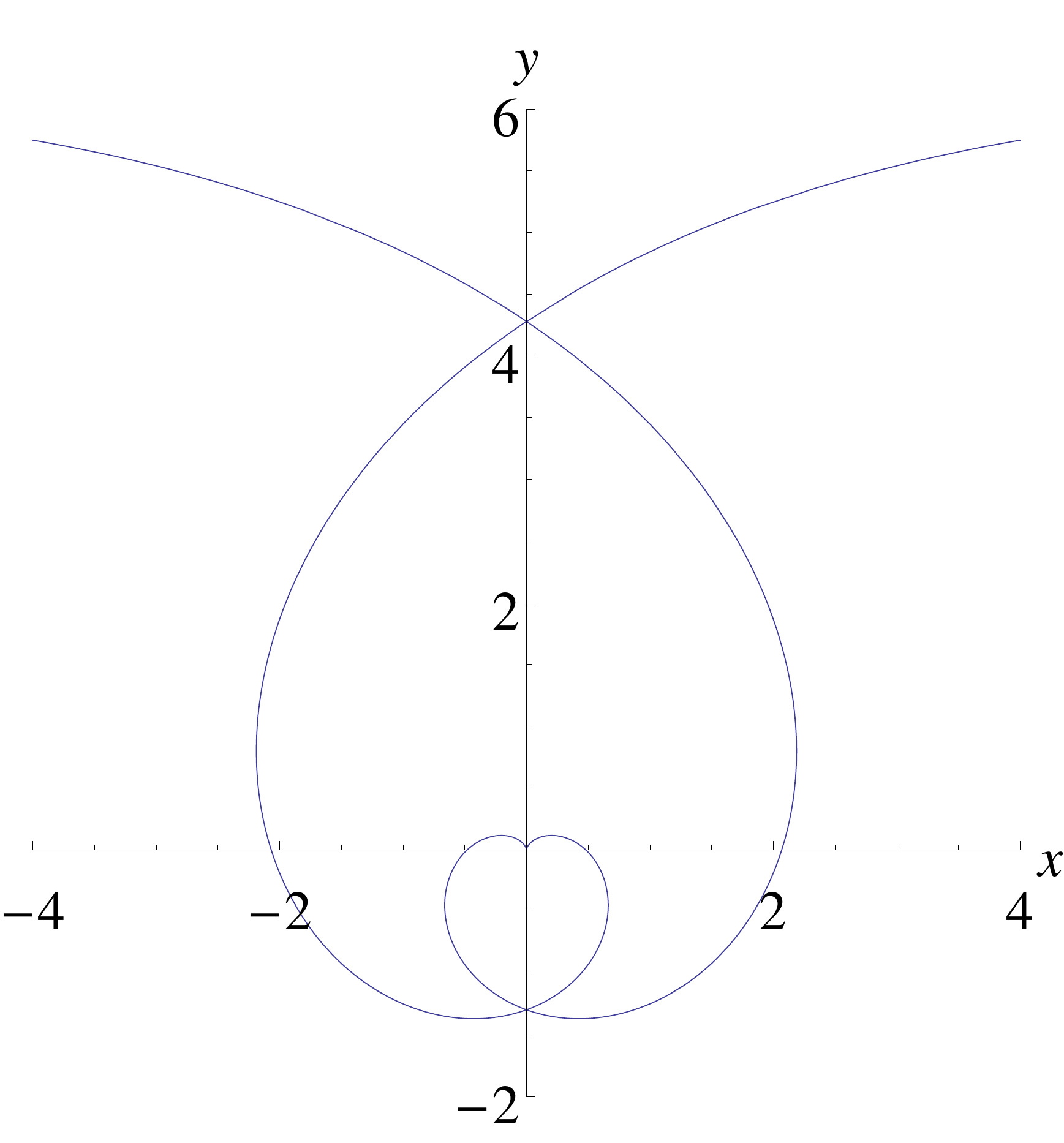}
    \caption{Trajectory plot associated to the orbit in
      Figure~\ref{fig:sleigh-phaseplot}. The sleigh approaches from
      the right.}
    \label{fig:sleigh-trajplot}
  \end{minipage}
\end{figure}

These equations give rise to the phase plot in
Figure~\ref{fig:sleigh-phaseplot}. A typical trajectory of the skate
point of contact is shown in Figure~\ref{fig:sleigh-trajplot}. We see
that the solutions are half-ellipses in the $(u,\omega)$-plane starting from the
negative $u$ axis and converging onto the positive $u$ axis. When $u$
is positive and $\omega = 0$, then the sleigh is moving in a straight
line with the center of mass forward of the skate. This turns out to
be the stable solution, while the opposite direction where $u < 0$ is
unstable.

\section{The sleigh with sliding friction}
\label{sec:sleigh-friction}

Now we shall obtain the Chaplygin sleigh
equations~\eqref{eq:sleigh-equations} in a different way. Instead of
enforcing the no-slip constraint $v = 0$ with the constraint reaction
force $F_c$, we consider the same Lagrangian without constraint, but
now we add a friction force $F_f$. Then we scale the friction force to
infinity and obtain the nonholonomic Chaplygin sleigh equations in the
limit. Note that the friction force $F_f$ is defined also outside the
constraint distribution $\Ddist$ --- and actually zero on it --- while
the reaction force $F_c$ is only defined on $\Ddist$!

This idea can be physically motivated in the following way. A
nonholonomic system with a no-slip constraint is a mathematical
idealization of a physical system where there is a (very) strong force
that prevents the system from going into slip. In this case, we assume
that a strong friction force prevents the skate from slipping
sideways. By showing that scaling the friction force to infinity leads
to nonholonomic dynamics, we provide a fundamental physics argument
for the d'Alembert principle of nonholonomic dynamics. We shall first
argue heuristically how this limit is obtained, and then make it
rigorous using geometric singular perturbation theory. Secondly, we can
consider what happens when the skate is turning too quickly at high
speed and the friction may not be able to generate the force necessary
to (nearly) keep the sleigh from slipping. Singular perturbation
theory also provides us with modifications to the nonholonomic
dynamics through a series expansion in the scaling parameter.

We start with the same Lagrangian~\eqref{eq:Lagr-sleigh}, but now insert
a friction force $F_f$ on the right-hand side. The friction force is
supposed to suppress sideways sliding of the skate blade, which is
given by the velocity $v=-\dot{x}\sin(\varphi)+\dot{y}\cos(\varphi)$.
We take the friction force linear in this slipping velocity and
pointing in the opposite direction to dampen it. Thus we
have\footnote{%
  We could have obtained this friction force more geometrically from a
  Rayleigh dissipation function
  $R(q,\dot{q}) = \frac{1}{2} \nu_q(\dot{q},\dot{q})$, where $\nu$ is
  a family of quadratic forms on $\T Q$ with kernel $\Ddist$.
  Then $F_f$ is given by minus the fiber derivative of $R$, that is,
  $F_f = -\pder{R}{\dot{q}}\colon \T Q/\Ddist \to \Ddist^0 \subset \T^* Q$.
  Using the moving frame coordinate $v$ to coordinatize
  $\T Q/\Ddist$, we find $R = \frac{1}{2} \bar{\nu}(q)\,v^2$. In our
  example we choose the friction coefficient $\bar{\nu}(q) = 1$.%
}
\begin{equation}\label{eq:sleigh-friction}
  F_f = -v \begin{pmatrix} -\sin(\varphi) \\ \ns\cos(\varphi) \end{pmatrix}.
\end{equation}
Note that this force acts at the point of contact of the skate, so
there is no associated torque around that point, i.e.\ the
$\omega$-component of the force is zero. Secondly, $F_f$ is zero when
$v = 0$ (as opposed to $F_c$), so this finite force does \emph{not}
enforce the nonholonomic constraint. We insert $F_f/\varepsilon$ into
the right-hand side of~\eqref{eq:sleigh-LdA}, so that we can scale the
friction force to infinity by taking the limit $\varepsilon \to 0$.
This yields
\begin{equation*}
  \begin{split}
    m\ddot{x} + ma\bigl(  -\ddot{\varphi}\sin(\varphi) -\dot{\varphi}^2\cos(\varphi)\bigr)
    &= \ns\frac{v}{\varepsilon} \sin(\varphi), \\[5pt]
    m\ddot{y} + ma\bigl(\ns\ddot{\varphi}\cos(\varphi) -\dot{\varphi}^2\sin(\varphi)\bigr)
    &=   -\frac{v}{\varepsilon} \cos(\varphi), \\[5pt]
    \begin{aligned}[b]
      (I + ma^2)\ddot{\varphi}
     +ma \der{}{t}   \Bigl[-\dot{x}\sin(\varphi)+\dot{y}\cos(\varphi)\Bigr]& \\[3pt]
   {}+ma\dot{\varphi}\bigl( \dot{x}\cos(\varphi)+\dot{y}\sin(\varphi)\bigr)&
   \end{aligned}
    &= 0.
  \end{split}
\end{equation*}
As in the previous section, we transform these equations to moving
frame coordinates, cf.~\eqref{eq:sleigh-LdA-frame}, and obtain
\begin{equation}\label{eq:sleight-friction-frame}
  \begin{split}
    \dot{u} - v\omega - a \omega^2    &= 0, \\[3pt]
    \dot{v} + u\omega + a\dot{\omega} &= -\frac{v}{m \varepsilon}, \\[3pt]
    (I + ma^2)\dot{\omega} + m a(\dot{v} + u \omega) &= 0,
  \end{split}
\end{equation}
but now we cannot insert the constraint condition $v = 0$.
Instead, we shall analyze the dynamics and see that when
$\varepsilon > 0$ is small, then $v(t)$ will quickly converge to near
zero. This will imply that the other two equations effectively behave
as if the constraint $v = 0$ is active. In other words, there is a
slow manifold described by the variables $(u,\omega)$ and on that
manifold the fast variable $v$ is small and slaved to $(u,\omega)$.
Secondly, the friction force naturally takes values in
$\Ddist^0$, the annihilator of $\Ddist$. This is preserved
in the limit to realize a constraint reaction force according to
d'Alembert's principle~\ref{hyp:dAlembert}. All of this indicates that
we can expect to recover nonholonomic dynamics.

To obtain a differential equation for $\dot{v}$, we subtract
$\frac{a}{I+ma^2}$ times the third equation from the second to make
the term with $\dot{\omega}$ cancel. This gives
\begin{equation}\label{eq:sleigh-fast-eq}
  \frac{I}{I + ma^2}(\dot{v} + u\omega) = -\frac{v}{m \varepsilon}.
\end{equation}
When $\varepsilon$ is small, this gives fast exponential
decay of $v$. We shall now assume that the typical rate at which
$v(t)$ changes is much faster than that of $u(t)$ and $\omega(t)$.
That is, we consider $v$ as fast variable and $u,\omega$ as slow
variables. Further conclusions based on this can be made rigorous
using singular perturbation theory, but we postpone these arguments to
the general theory in Section~\ref{sec:general} and focus on obtaining
the result.

Thus, we assume in~\eqref{eq:sleigh-fast-eq} that $u,\omega$ are
approximately constant and obtain as solution
\begin{equation}\label{eq:sleigh-fast-sol}
  v(t) = \Bigl(v(0) + \frac{u \omega}{\rho}\Bigr) e^{-\rho t} - \frac{u \omega}{\rho}
  \qquad\text{with } \rho = \frac{I + ma^2}{I m \varepsilon}.
\end{equation}
This quickly settles to $v = - \frac{u \omega}{\rho}$, so for
the slow dynamics of $u,\omega$ we insert this relation, which gives
\begin{equation}\label{eq:sleigh-slow-eq}
  \dot{u} = a \omega^2 + \varepsilon\frac{I m}{I + ma^2} u \omega^2, \qquad
  \dot{\omega} = -\frac{m a u \omega}{I + ma^2}
                 +\varepsilon I m^2 a \der{}{t}\bigl[u \omega\bigr].
\end{equation}
In these differential equations $\varepsilon$ only appears in the
numerator and the $\varepsilon$ multiplying the time derivative of
$(u \omega)$ does not lead to singular equations, so we can sensibly
take the limit: we simply insert $\varepsilon = 0$ to arrive at the
original equations~\eqref{eq:sleigh-equations} for the nonholonomically
constrained Chaplygin sleigh. This is also confirmed by numerical
integration of this system: with decreasing values of $\varepsilon$,
the trajectories converge to the trajectory of the nonholonomic
system, see Figure~\ref{fig:sleigh-frictrajplots}. It is clearly
visible though that the convergence is not uniform in time, as already
illustrated in Section~\ref{sec:diff-constraints}. In
Section~\ref{sec:sleigh-expand} we will analyze the perturbation to
the slow system in more detail.

\begin{figure}[htb]
  \centering
  \includegraphics[width=9cm]{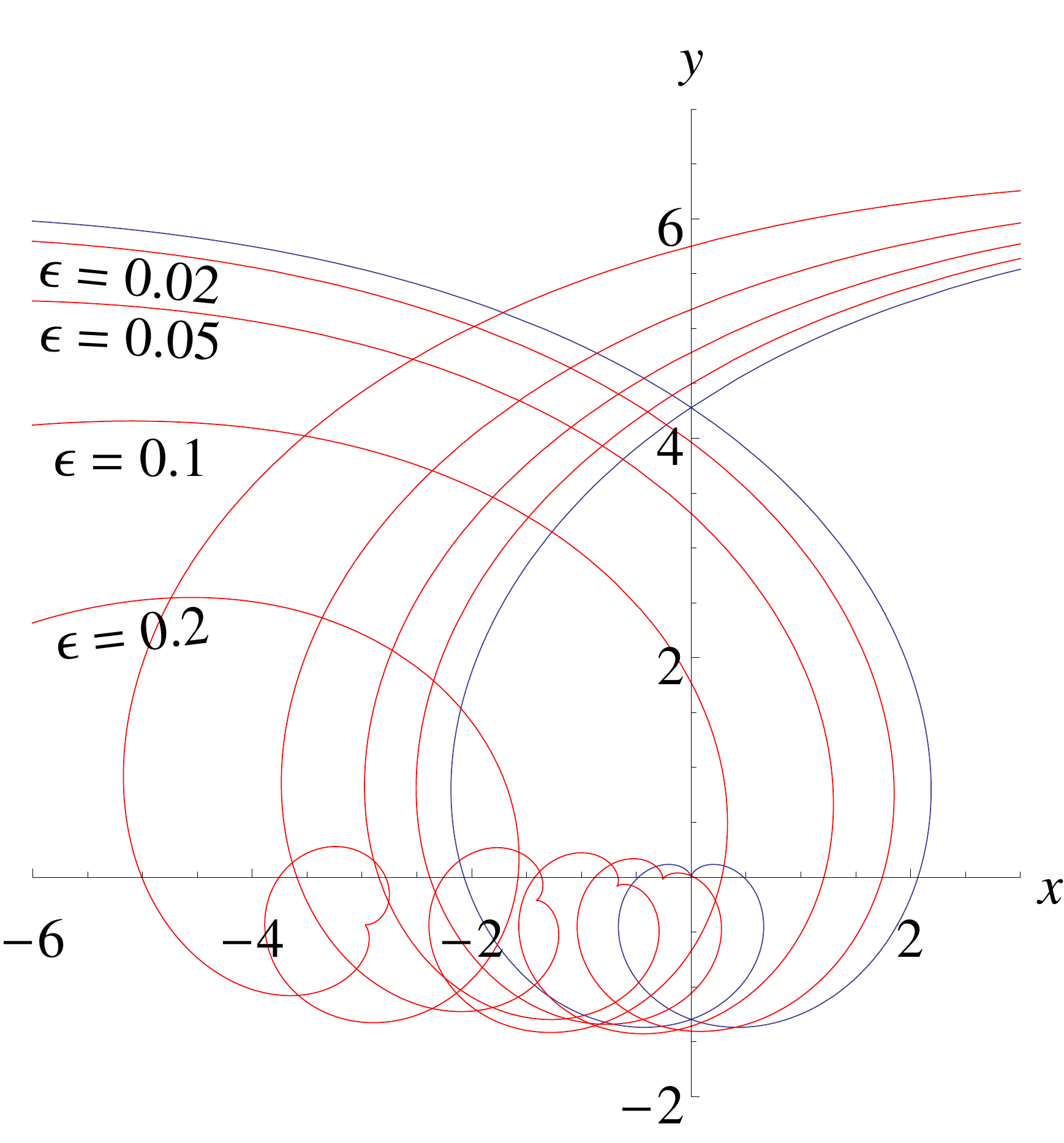}
  \caption{The red orbits are trajectories of the sleigh with friction
    with indicated parameter values of $\varepsilon$. These clearly
    converge to the nonholonomic trajectory in blue.}
  \label{fig:sleigh-frictrajplots}
\end{figure}

\section{The general theory}
\label{sec:general}

We shall finally prove the statement that linear friction forces
realize nonholonomic dynamics in a more general setting. This
rephrases the results in~\cite{Brendelev1981:realnonhol,
  Karapetian1981:realnonhol} in a more geometric formulation, and
specifies in what sense the solution curves converge, also on
positively unbounded time intervals.

Recall that we consider Lagrangian systems of mechanical type,
\begin{equation}\tag{\ref{eq:Lagrangian}}
  L(q,\dot{q}) = \frac{1}{2}\kappa_q(\dot{q},\dot{q}) - V(q),
  \qquad (q,\dot{q}) \in \T Q.
\end{equation}
Furthermore, to state our result, we need the concept of a
pseudo solution (or pseudo orbit) of a dynamical system.
\begin{definition}[Pseudo solution]\label{def:pesudo-sol}
  Let $(M,g)$ be a smooth Riemannian manifold and let
  $X \in \mathfrak{X}(M)$ be a $C^1$ vector field on it. We say that a
  $C^1$ curve $x(t) \in M$ is a $\delta$-pseudo solution of $X$ when
  \begin{equation}\label{eq:pseudo-sol}
    \norm{\dot{x}(t) - X(x(t))} \le \delta \qquad\text{for all $t$.}
  \end{equation}
\end{definition}
Furthermore, we consider linear friction forces $F$ that are specified
by a Rayleigh function $R$. That is, given a smooth family of positive
semi-definite bilinear forms
$\nu_q\colon \T_q Q \times \T_q Q \to \R$, we define the Rayleigh
function $R(q,\dot{q}) = \frac{1}{2}\nu_q(\dot{q},\dot{q})$ and the
friction force as its negative fiber derivative,
\begin{equation}\label{eq:Rayleigh-friction}
  F\colon \T Q \to \T^* Q,\quad
  F(q,\dot{q}) = -\pder{R}{\dot{q}}(q,\dot{q}) = -\nu_q^\flat(\dot{q}).
\end{equation}

Then we have the following result\footnote{%
  We shall assume that everything has sufficient smoothness, say,
  $C^r$ for a large $r$. Note that the perturbed manifolds
  $\Ddist_\varepsilon$ will generally only have finite degree of
  smoothness, see e.g.~\cite[Sect.~1.2.1]{Eldering2013:NHIM-noncompact}.%
}
\begin{theorem}\label{thm:friction-limit}
  Let $Q$ be a compact manifold, let $L$ be a Lagrangian of mechanical
  type~\eqref{eq:Lagrangian}, and let $F$ be a linear friction force
  with kernel $\ker(F) = \Ddist$ a regular distribution. Let
  $X_\varepsilon \in \mathfrak{X}(\T Q)$ denote the vector field of
  the Lagrangian system with Rayleigh friction force $F/\varepsilon$
  as above, and let $X_\textup{NH} \in \mathfrak{X}(\Ddist)$ denote
  the vector field of the Lagrange--d'Alembert nonholonomic system
  $(Q,L,\Ddist)$.

  Then solution curves of $X_\varepsilon$ converge to solutions of
  $X_\textup{NH}$ as $\varepsilon \to 0$ in the following sense. Fix
  an energy level $\bar{E} > \inf V$ and consider initial
  conditions $x_0 \in \T Q$ with energy less than $\bar{E}$. Then for
  all $\varepsilon > 0$ sufficiently small we have the following
  results.
  \begin{enumerate}
  \item All solution curves $x_\varepsilon(t)$ of $X_\varepsilon$
    converge at a uniform exponential rate to an invariant
    manifold\footnote{%
      The manifold $\Ddist_\varepsilon$ is unique up to the choice of
      a cutoff function at the boundary of the energy $\bar{E}$ in
      $\T Q$. Different choices lead to differences of order
      $\varepsilon$ that decay exponentially away from the boundary.
      However, the limit to the reduced vector field $X_\textup{NH}$
      will not depend on this choice.%
    } $\Ddist_\varepsilon$ that is $C^1$-close and diffeomorphic to
    $\Ddist$. The manifold $\Ddist_\varepsilon$ depends smoothly on
    $\varepsilon$ near zero, with $\Ddist_0 = \Ddist$.

  \item There exists a family of curves
    $\bar{x}_\varepsilon(t) \in \Ddist$ that are
    $\mathcal{O}(\varepsilon)$-pseudo solution curves of
    $X_\textup{NH}$, such that\footnote{%
      Here $d$ is some appropriate distance on $\T Q$, for example the
      distance induced by the Sasaki metric coming from $(Q,\kappa)$.%
    }
  \begin{equation}\label{eq:conv-to-pseudo-sols}
    \forall t_1 > 0 \quad
    \sup_{t \in [t_1,\infty)} d\bigl(x_\varepsilon(t),\bar{x}_\varepsilon(t)\bigr) \to 0
  \end{equation}
  as $\varepsilon \to 0$, uniformly in $x_0$.
  \end{enumerate}
\end{theorem}

The compactness condition of $Q$ can be replaced by uniformity
assumptions on all of $Q,L,F,\Ddist$. One crucial assumption that has
to be added is that $V$ is bounded from below. Moreover, one should
require $(Q,\kappa)$ to have bounded geometry and $\d V$, $F$, and
$\Ddist$ to be uniformly $C^1$ bounded and $F$ also bounded away from
zero. This is needed if one wants to apply normal hyperbolicity theory
in a noncompact context,
see~\cite[Thm.~3.1]{Eldering2013:NHIM-noncompact}. In the Chaplygin
sleigh example $Q = \R^2 \times S^1$ is not compact, but this
generalization does apply: the system is symmetric under the Euclidean
group $\mathrm{SE}(\R^2)$ acting transitively on $Q$, hence
$(Q,\kappa)$ has bounded geometry
(see~\cite[Ex.~2.3]{Eldering2013:NHIM-noncompact}) and the remaining
uniformity conditions are fulfilled as well. Another viewpoint is to
compactify $\R^2$ to the two-torus, or to consider the closely related
Suslov problem of a nonholonomically constrained rigid body, so that
Theorem~\ref{thm:friction-limit} does apply.

Alternatively, one
can consider convergence of solution curves only on finite time
intervals, which effectively allows one to restrict to a compact set
covering the nonholonomic solution curve. This is essentially the
setting of the results in~\cite{Brendelev1981:realnonhol,
  Karapetian1981:realnonhol}, which are given in local coordinates. It
does not automatically yield uniformity of convergence of solution
curves (even on finite time intervals) when $Q$ is noncompact though,
as illustrated by the following example.

Let $Q = \R^2$ with the Euclidean metric and consider a particle of
unit mass under the potential $V(x,y) = -x^2$. Let the constraint
distribution be
\begin{equation*}
  \Ddist = \operatorname{span} \chi(x-y)\partial_x + \bigl(1-\chi(x-y)\bigr)\partial_y,
\end{equation*}
where $\chi \in C^\infty(\R;[0,1])$
such that $\chi(x) = 0$ for $x \le 0$ and $\chi(x) = 1$ for $x \ge 1$.
Finally, let the friction force $F$ be arbitrary, but with kernel
$\Ddist$ and uniformly bounded.
Now consider the unconstrained system with friction and initial
conditions $x(0) > 0$ small, $y(0) \ge 1$ arbitrarily large and
$\dot{x}(0) = \dot{y}(0) = 0$. Then the particle will first accelerate
in the positive $x$ direction due to the potential; once it hits the
line $x = y$, the friction force kicks in to make the particle follow
the changing constraint direction of $\Ddist$. However, for any fixed
$\varepsilon > 0$, the particle's velocity component $\dot{x}$ will
overshoot the constraint $\Ddist = \operatorname{span} \partial_y$ in
the region $x \ge y + 1$ by an arbitrary amount as the initial
condition $y(0)$ is made large, since $\dot{x}$ then is arbitrary
large at `impact' with the changing constraint. Even worse, since the
potential force increases with $x$, even in the region $x \ge y + 1$,
the velocity component $\dot{x}$ will diverge away from zero, that is,
from $\Ddist$.

\begin{proof}[Proof of Theorem~\ref{thm:friction-limit}]
The system with friction described by $X_\varepsilon$ is given by the
second order equations
\begin{equation}\label{eq:EL-friction}
  \kappa^\flat(\nabla_{\dot{q}} \dot{q})
  = -\d V(q) + \frac{1}{\varepsilon} F(q,\dot{q}),
\end{equation}
or, in explicit local coordinates it is expressed as
\begin{equation}\label{eq:EL-friction-coords}
  \dot{q}^i = v^i, \qquad
  \dot{v}^i = -\Gamma^i_{jk} v^j v^k
             + \kappa^{ij}\Bigl(-\pder{V}{q^j} + \frac{1}{\varepsilon}F_j\Bigr).
\end{equation}
The limit of the vector field $X_\varepsilon$ is singular, so we
switch to a rescaled version
$Y_\varepsilon = \varepsilon\,X_\varepsilon$. This can be viewed as
changing to a fast time variable $\tau = t/\varepsilon$ (whose
derivative we denote by a prime) for any $\varepsilon > 0$, but the
vector field $Y_0$ is also well-defined and given in local coordinates
as
\begin{equation}\label{eq:Y0}
  q'^i = 0, \qquad
  v'^i = \kappa^{ij} F_j = -\kappa^{ij}\,\nu_{jk}\,v^k.
\end{equation}
Both $\kappa^{ij}$ and $\nu_{jk}$ are (semi-)positive definite
matrices with $\ker(\nu_q) = \Ddist_q$. This shows that the
smooth manifold $\Ddist \subset \T Q$ is the fixed point set of
$Y_0$ and that $Y_0$ is linear in the normal direction, and the
linear term given by
\begin{equation*}
  -\kappa^\sharp \circ \nu^\flat\big|_{\T Q / \Ddist}
\end{equation*}
has strictly negative eigenvalues when we identify
$\T Q / \Ddist \cong \Ddist^\perp$. Thus, $\Ddist$ is a normally
hyperbolic invariant manifold (\NHIM) for $Y_0$ and it is globally
exponentially attractive.

We now restrict attention to the compact subset
\begin{equation*}
  \mathcal{E} = \{ (q,v) \in \T Q \mid E(q,v) \le \bar{E} \}
\end{equation*}
below the energy level $\bar{E}$, which is nonempty when
$\bar{E} > \inf V$. The dynamics of $X_\varepsilon$ (and thus also of
$Y_\varepsilon$) is dissipative, hence leaves $\mathcal{E}$ forward
invariant. This implies that we can use a cutoff argument\footnote{%
  Without going into full details: we modify $Y_\varepsilon$ at the
  boundary of $\mathcal{E}$ such that $\Ddist \cap \mathcal{E}$
  is overflowing invariant (see~\cite{Fenichel1971:invarmflds}) for
  the modified vector field.%
} outside $\mathcal{E}$ which will not affect forward solutions with
initial conditions inside $\mathcal{E}$. Using persistence of \NHIM{}s
(see~\cite{Fenichel1971:invarmflds, Hirsch1977:invarmflds,
  Eldering2013:NHIM-noncompact}) we find
that for $\varepsilon > 0$ sufficiently small, there is a unique
perturbed manifold $\Ddist_\varepsilon$ that is $C^1$-close and
diffeomorphic to $\Ddist$ and invariant under $Y_\varepsilon$.
Moreover, $\Ddist_\varepsilon$ depends smoothly on
$\varepsilon$ and is a \NHIM again with perturbed stable fibration. In
particular, $\Ddist_\varepsilon$ is locally exponentially
attractive, while the exponential attraction globally inside
$\mathcal{E}$ follows from uniform smooth dependence of flows on
parameters over compact domains and time intervals, see
e.g.~\cite[Thm.~A.6]{Eldering2013:NHIM-noncompact}. Hence solution
curves $\tilde{x}_\varepsilon(\tau)$ of $Y_\varepsilon$ converge at a
uniform exponential rate to $\Ddist_\varepsilon$, and thus the
solution curves $x_\varepsilon(t) = \tilde{x}_\varepsilon(t/\varepsilon)$
of $X_\varepsilon$ do so too, and at increasing rates as
$\varepsilon \to 0$. This proves the first part of the claims.

To prove the remaining claims, we need to analyze the dynamics of
$Y_\varepsilon$ restricted to the invariant manifold
$\Ddist_\varepsilon$. Let us denote by
\begin{equation}\label{eq:proj-D}
  \pr\colon \T Q \to \Ddist
\end{equation}
the projection along $\Ddist^\perp$, that is, orthogonal
according to the kinetic energy metric $\kappa$. Note
that~\eqref{eq:proj-D} can be viewed as a vector bundle where we
forget about the linear structure on $\Ddist$. The
Lagrange--d'Alembert equations~\eqref{eq:LdA-equations} can then be
expressed as
\begin{equation*}
  \pr \circ [L](\ddot{q}(t)) = 0
  \qquad\text{with } \dot{q}(t) \in \Ddist.
\end{equation*}
Furthermore, let $f\colon U \times \R^n \to \T Q$ be a frame on an
open set $U \subset Q$, such that the first $k$ components span
$\Ddist$ and the remaining $n\!-\!k$ span $\Ddist^\perp$. Let
$\omega \in \Omega^1\bigl(Q;\textup{End}(\R^n)\bigr)$ denote the
connection one-form with respect to the frame $f$ and let
$(\xi,\eta) \in \R^n$ denote coordinates associated to this
frame, where $\xi \in \R^k$ and $\eta \in \R^{n-k}$ respectively
coordinatize the fibers of $\Ddist$ and $\Ddist^\perp$. In these frame
coordinates, the projection $\pr$ simply maps $(\xi,\eta)$
to $(\xi,0)$. Let $\pr_\xi\colon \R^n \to \R^k$ denote this
projection, $\pr_\eta\colon \R^n \to \R^{n-k}$ its complementary
projection. See Appendix~\ref{sec:conn-form-struc-func} for a brief
overview of using a connection one-form to express equations of motion
in a moving frame and its relation to using structure functions.

In frame coordinates the Lagrange--d'Alembert equations are then
\begin{equation}\label{eq:LdA-frame}
  X_\textup{NH} \Longleftrightarrow \left\{
  \begin{aligned}
    \dot{q}   &= f_q \cdot (\xi,0),\\
    \dot{\xi} &= \pr_\xi \Bigl[ -\omega(\dot{q}) \cdot (\xi,0)
                                -f_q^{-1} \cdot \kappa_q^\sharp \cdot \d V\Bigr],
  \end{aligned}\right.
\end{equation}
since $\dot{q} \in \Ddist_q$ is equivalent to $\eta = 0$. This is
the vector field $X_\textup{NH}$ on $\Ddist$.

Next, we show that there is a well-defined limit of the family of
vector fields
\begin{equation*}
  X_\varepsilon|_{\Ddist_\varepsilon} \in \mathfrak{X}(\Ddist_\varepsilon)
\end{equation*}
and that the limit is precisely $X_\textup{NH}$. To make sense of this
limit, however, we first have to re-express these vector fields on a
fixed manifold, $\Ddist$. This we do using the projection
$\pr\colon \T Q \to \Ddist$. The invariant manifolds are $C^1$-close
and diffeomorphic to $\Ddist$, so they can be written as the graphs of
functions $h_\varepsilon\colon \Ddist \to \Ddist^\perp$ covering the
identity in $Q$, or in other words as graphs of sections
$h_\varepsilon \in \Gamma(\pr\colon \T Q \to \Ddist)$. That is, in
frame coordinates we have
\begin{equation}\label{eq:Deps-graph}
  \Ddist_\epsilon = \{\, \eta = h_\varepsilon(q,\xi) \,\}.
\end{equation}
Since $\Ddist_\varepsilon$ depends smoothly on $\varepsilon$, we
can consider a Taylor expansion of $h_\varepsilon$. This we denote as
\begin{equation}\label{eq:graph-expansion}
  h_\varepsilon(q,\xi) = \sum_{i=0}^r \frac{\varepsilon^i}{i!}\,h^{(i)}(q,\xi)
                         + \mathcal{O}(\varepsilon^{r+1}).
\end{equation}
Note that $h^{(0)}(q,\xi) = h_0(q,\xi) = 0$ since $\Ddist_0 = \Ddist$.
We now consider $Y_\varepsilon$ in frame coordinates,
cf.~\eqref{eq:EL-friction} and~\eqref{eq:LdA-frame}. Let us denote by
$\bigl(\kappa_q^\sharp\,\nu_q^\flat\bigr)_{\!f}$ the linear operator
$\kappa_q^\sharp\,\nu_q^\flat\colon \Ddist_q^\perp \to \Ddist_q^\perp$
with respect to the frame $f$, then
\begin{equation}\label{eq:EL-friction-frame}
  Y_\varepsilon \Longleftrightarrow \left\{
  \begin{aligned}
    q'    &= \varepsilon\, f_q \cdot (\xi,\eta),\\
    \xi'  &= \varepsilon\,\pr_\xi\, \Bigl[
                -\omega\bigl(f_q \cdot (\xi,\eta)\bigr) \cdot (\xi,\eta)
                -f_q^{-1} \cdot \kappa_q^\sharp \cdot \d V \Bigr],\\
    \eta' &= \varepsilon\,\pr_\eta \Bigl[
                -\omega\bigl(f_q \cdot (\xi,\eta)\bigr) \cdot (\xi,\eta)
                -f_q^{-1} \cdot \kappa_q^\sharp \cdot \d V \Bigr]
             -\bigl(\kappa_q^\sharp\,\nu_q^\flat\bigr)_{\!f} \cdot \eta.
  \end{aligned}\right.
\end{equation}
Note that $\omega(f_q \cdot \;)$ are the connection coefficients  as
defined in Appendix~\ref{sec:conn-form-struc-func}, but written
without indices. Since $Y_\varepsilon$ leaves $\Ddist_\varepsilon$
invariant, we can consider the restriction
$Y_\varepsilon|_{\Ddist_\varepsilon}$. Secondly,
$\pr\colon \Ddist_\varepsilon \to \Ddist$ is a
diffeomorphism (its inverse is the map
$\textup{Id}_\Ddist + h_\varepsilon$), so we can push forward the
vector field $Y_\varepsilon|_{\Ddist_\varepsilon}$ along $\pr$ to
\begin{equation*}
  \tilde{Y}_\varepsilon
  := \pr_*\bigl(Y_\varepsilon|_{\Ddist_\varepsilon}\bigr)
  \in \mathfrak{X}(\Ddist).
\end{equation*}
In frame coordinates this amounts to projecting the vector field onto
the coordinates $(q,\xi)$, while inserting
$\eta = h_\varepsilon(q,\xi)$, see Figure~\ref{fig:invar-mflds}. That
is, in frame coordinates we have
\begin{equation}\label{eq:EL-friction-invar}
  \tilde{Y}_\varepsilon \Longleftrightarrow \left\{
  \begin{aligned}
    q'   &= \varepsilon\, f_q \cdot (\xi,h_\varepsilon(q,\xi)),\\
    \xi' &= \varepsilon\,\pr_\xi\, \Bigl[
               -\omega\bigl(f_q \cdot (\xi,h_\varepsilon(q,\xi))\bigr) \cdot (\xi,h_\varepsilon(q,\xi))
               -f_q^{-1} \cdot \kappa_q^\sharp \cdot \d V \Bigr].
  \end{aligned}\right.
\end{equation}
\begin{figure}[htb]
  \centering
  \input{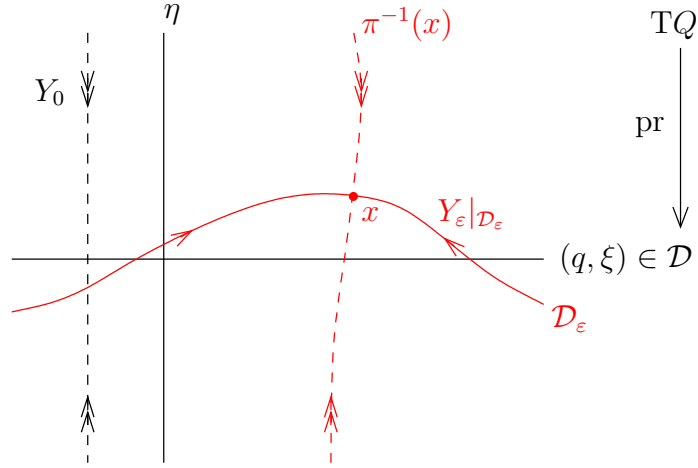}
  \caption{The invariant manifolds $\Ddist_\varepsilon$ and vector
    fields living on them in frame coordinates $(q,\xi,\eta)$. The
    dashed lines indicate stable fibers.}
  \label{fig:invar-mflds}
\end{figure}

Now we note that $\tilde{Y}_\varepsilon$ is of order $\varepsilon$, so
this vector field can be rescaled back to the original slow time
$t = \varepsilon\,\tau$, which we denote as
$\tilde{X}_\varepsilon = \frac{1}{\varepsilon}\tilde{Y}_\varepsilon$.
Then we can consider its limit
\begin{equation}\label{eq:X-friction-limit}
  \tilde{X}_0
  := \lim_{\varepsilon \to 0} \tilde{X}_\varepsilon \in \mathfrak{X}(\Ddist)
\end{equation}
and we find $\tilde{X}_0 = X_\textup{NH}$, that is, the
Lagrange--d'Alembert equations: we have
$h_\varepsilon(q,\xi) = \varepsilon\,h^{(1)}(q,\xi) + \mathcal{O}(\varepsilon^2)$,
so after dividing~\eqref{eq:EL-friction-invar} by $\varepsilon$ and
inserting $\eta = h_\varepsilon(q,\xi)$, the only terms remaining in
the limit are those with $\eta = 0$. This yields exactly the
nonholonomic vector field $X_\textup{NH}$ as in~\eqref{eq:LdA-frame}.

Using this limit, we can finally prove the last dynamical statement
that solution curves $x_\varepsilon(t)$ of $X_\varepsilon$ converge to
curves $\bar{x}_\varepsilon(t) \in \Ddist$ that are uniformly
$\mathcal{O}(\varepsilon)$-pseudo solutions of $X_\textup{NH}$. Using
persistence of normal hyperbolicity we already proved that
$x_\varepsilon(t)$ converges to the invariant manifold
$\Ddist_\varepsilon$ uniformly in $x_0$. Moreover,
$\Ddist_\varepsilon$ has a stable fibration whose projection commutes
with the flow of $X_\varepsilon$. That is, there exists a (nonlinear)
fibration $\pi\colon \T Q \to \Ddist_\varepsilon$ such that
$\pi(\Phi_\varepsilon^t(x_0)) = \Phi_\varepsilon^t(\pi(x_0))$. Each
single fiber $\pi^{-1}(x)$ is not invariant, but the flow does map
fibers onto fibers and is exponentially contracting along them. Hence,
$x_\varepsilon(t)$ projects to a solution curve
$\pi(x_\varepsilon(t)) \in \Ddist_\varepsilon$ to which it converges
as $t \to \infty$. Now define
$\bar{x}_\varepsilon(t) := \pr(\pi(x_\varepsilon(t))) \in \Ddist$,
which is by construction a solution of $\tilde{X}_\varepsilon$. Since
$\tilde{X}_\varepsilon - X_\textup{NH} \in \mathcal{O}(\varepsilon)$
and $\norm{h_\varepsilon}_\textrm{sup} \in \mathcal{O}(\varepsilon)$
uniformly on the compact set $\mathcal{E}$, it follows that
$\bar{x}_\varepsilon(t)$ is a $\mathcal{O}(\varepsilon)$-pseudo
solution of $X_\textup{NH}$ and that $\pi(x_\varepsilon)$ and
$\bar{x}_\varepsilon$ are
$\mathcal{O}(\varepsilon)$-close uniformly in time. Finally, since the
fast-time-$\tau$ solutions $\tilde{x}_\varepsilon(\tau)$ already
converge to $\Ddist_\varepsilon$ at a fixed exponential rate uniformly
in $x_0 \in \mathcal{E}$, there exists a $\tau_1 > 0$ such that
$d\bigl(\tilde{x}_\varepsilon(\tau),\pi(\tilde{x}_\varepsilon(\tau))\bigr) < \varepsilon$
for all $\tau \ge \tau_1$. Now consider $t_1 > 0$ arbitrary. Since
$x_\varepsilon(t) = \tilde{x}_\varepsilon(t/\varepsilon)$ we see that
when $\varepsilon \le t_1 / \tau_1$, also
$d\bigl(x_\varepsilon(t),\pi(x_\varepsilon(t))\bigr) < \varepsilon$
for all $t \ge t_1$, and finally
\begin{equation*}
  d(x_\varepsilon(t),\bar{x}_\varepsilon(t))
  \le d\bigl(x_\varepsilon(t),\pi(x_\varepsilon(t))\bigr) +
      d\bigl(\pi(x_\varepsilon(t)),\bar{x}_\varepsilon(t)\bigr)
  \in \mathcal{O}(\varepsilon).
\end{equation*}
This completes the proof of the second claim.
\end{proof}

\section{Beyond the limit of nonholonomic dynamics}
\label{sec:beyond-limit}

The Lagrange--d'Alembert equations were obtained as zeroth order term
in the expansion of $\tilde{X}_\varepsilon$ in $\varepsilon$. However,
one can continue and inductively find higher order terms in the
expansion of $\tilde{X}_\varepsilon$. These terms correspond to
effects of large, but finite friction forces, and will for example
contribute to drift normal to the nonholonomic constraint and energy
dissipation. The advantage of studying large, but finite friction in
this context is that $\tilde{X}_\varepsilon$ is still a vector field
on the lower-dimensional nonholonomic phase space $\Ddist$ as compared
to the full Euler--Lagrange equations on $\T Q$, while normal
hyperbolicity guarantees that this is a proper Taylor expansion of
a truly invariant and attractive subsystem. We shall here show how
these higher terms can be obtained and calculate the first order
correction term.

The `master equation' for obtaining the expansion is derived from the
coordinate expression for the invariant manifold $\Ddist_\varepsilon$,
given by $\eta = h_\epsilon(q,x)$. Let $f_\varepsilon(q,\xi,\eta)$
denote the `horizontal' components of $Y_\varepsilon$ for $(q',\xi')$
and let $g_\varepsilon(q,\xi,\eta)$ describe the `vertical' component
for $\eta'$. Taking a fast-time-$\tau$ derivative of the invariant
manifold equations, we find
\begin{align}
  \eta' &= \D h_\varepsilon(q,\xi) \cdot (q',\xi') \notag \\
  \Longleftrightarrow\quad
  g_\varepsilon(q,\xi,h_\varepsilon(q,\xi))
        &= \D h_\varepsilon(q,\xi) \cdot f_\varepsilon(q,\xi,h_\varepsilon(q,\xi)).
  \label{eq:Ddist-expand-master}
\end{align}
This equation we can now expand in powers of $\varepsilon$ to obtain
inductively the functions $h^{(i)}(q,\xi)$. Before starting this, let
us note the following: using the same expansion as
in~\eqref{eq:graph-expansion} for $f_\varepsilon$ and $g_\varepsilon$
we have for all $(q,\xi,\eta)$ that
\begin{alignat*}{2}
  f^{(0)}(q,\xi,\eta) &= 0, &\qquad h^{(0)}(q,\xi) &= 0, \\
  g^{(0)}(q,\xi,0) &= 0, &\qquad
  \D_3 g^{(0)}(q,\xi,0) &= -\bigl(\kappa_q^\sharp\,\nu_q^\flat\bigr)_{\!f}
\end{alignat*}
and $\D_3 g^{(0)}(q,\xi,0)$ is invertible as a linear map on
$\R^{n-k}$. Expanding~\eqref{eq:Ddist-expand-master} up to first
order in $\varepsilon$ then leads to
\begin{equation*}
  g^{(0)}(q,\xi,h^{(0)}(q,\xi)) = \D h^{(0)}(q,\xi) \cdot
  f^{(0)}(q,\xi,h^{(0)}(q,\xi))
\end{equation*}
at zeroth order, and is trivially satisfied as it
reduces to $0 = 0$. At first order we find
\begin{multline*}
  g^{(1)}(q,\xi,h^{(0)}(q,\xi)) + \D_3 g^{(0)}(q,\xi,h^{(0)}(q,\xi)) \cdot h^{(1)}(q,\xi) \\
  = \D h^{(1)}(q,\xi) \cdot f^{(0)}(q,\xi,h^{(0)}(q,\xi)) \\
    +\D h^{(0)}(q,\xi) \cdot \Bigl[
       f^{(1)}(q,\xi,h^{(0)}(q,\xi))
      +\D_3 f^{(0)}(q,\xi,h^{(0)}(q,\xi)) \cdot h^{(1)}(q,\xi) \Bigr]
\end{multline*}
which greatly simplifies to
\begin{equation*}
  g^{(1)}(q,\xi,0) + \D_3 g^{(0)}(q,\xi,0) \cdot h^{(1)}(q,\xi) = 0.
\end{equation*}
Using invertibility of $\D_3 g^{(0)}$, this can be solved for
$h^{(1)}$ and gives
\begin{equation}\label{eq:h-first-order}
  \begin{split}
    h^{(1)}(q,\xi)
    &= - \bigl[\D_3 g^{(0)}(q,\xi,0)\bigr]^{-1} \cdot g^{(1)}(q,\xi,0) \\
    &= (\nu_q^\sharp \, \kappa_q^\flat)_{\!f} \cdot \pr_\eta \cdot
       \bigl[ -\omega(f_q \cdot \xi) \cdot \xi - f_q^{-1} \cdot \kappa_q^\sharp \cdot \d V \bigr],
  \end{split}
\end{equation}
where we consider $\nu$ restricted to $\Ddist$. The term
$(\kappa_q^\flat)_{\!f} \cdot \pr_\eta \cdot \bigl[ -\omega(f_q \cdot \xi) \cdot
\xi - f_q^{-1} \cdot \kappa_q^\sharp \cdot \d V \bigr]$
can be identified as minus the reaction force
$F_c\colon \Ddist \to \Ddist^0$ in the nonholonomic constraint
picture, while $\eta = h^{(1)}(q,\xi)$ is the drift velocity violating
the nonholonomic constraint that is necessary to generate the same
force as $F_c$, but now due to the linear friction $F_f$ and up to
order $\varepsilon$.

To finally obtain the first order expansion of
$\tilde{X}_\varepsilon$, we have to expand
$f_\varepsilon(q,\xi,h_\varepsilon(q,\xi))$ up to order two. Here we
find
\begin{equation}\label{eq:f-second-order}
  \begin{split}
    \frac{1}{2} \left.\left(\der{}{\varepsilon}\right)^2\right|_{\varepsilon=0}
      f_\varepsilon(q,\xi,h_\varepsilon(q,\xi))
    &= \begin{aligned}[t]
         & \frac{1}{2}f^{(2)}(q,\xi,0) + \D_3 f^{(1)}(q,\xi,0) \cdot h^{(1)}(q,\xi) \\
         &+\frac{1}{2}\D_3 f^{(0)}(q,\xi,0) \cdot h^{(2)}(q,\xi)
          +\frac{1}{2}\D_3^2 f^{(0)}(q,\xi,0) \cdot h^{(1)}(q,\xi)^2
       \end{aligned}\\
    &= \D_3 f^{(1)}(q,\xi,0) \cdot h^{(1)}(q,\xi),
  \end{split}
\end{equation}
noting that also $f^{(2)} \equiv 0$. From~\eqref{eq:EL-friction-invar}
we read off that
\begin{equation*}
  \D_3 f^{(1)}(q,\xi,0) \cdot \eta
  = \Bigl( f_q \cdot \eta \, , \,
           -\pr_\xi\bigl[ \omega(f_q \cdot \eta) \cdot \xi
                         +\omega(f_q \cdot \xi ) \cdot \eta \bigr] \Bigr).
\end{equation*}
Putting everything together, we obtain
$\tilde{X}_\varepsilon = X_\textup{NH} + \varepsilon\,\tilde{X}^{(1)}
+ \mathcal{O}(\varepsilon^2)$ with
\begin{equation}\label{eq:X-first-order}
  \tilde{X}^{(1)} \Longleftrightarrow \left\{
  \begin{aligned}
    \dot{q}   &= f_q \cdot h^{(1)}(q,\xi) \\
    \dot{\xi} &= -\pr_\xi\, \Bigl[ \omega(f_q \cdot h^{(1)}(q,\xi)) \cdot \xi
                                  +\omega(f_q \cdot \xi) \cdot h^{(1)}(q,\xi) \Bigr]
  \end{aligned}\right.
\end{equation}
Note that with this correction term, $\tilde{X}_\varepsilon$ is not a
second order vector field on $\Ddist$ anymore since the drift velocity
$h^{(1)}(q,\xi)$ violates the nonholonomic constraint $\Ddist$.

\subsection{The Chaplygin sleigh}\label{sec:sleigh-expand}

We shall now follow this recipe to obtain a first order perturbation
for the nonholonomic dynamics of the Chaplygin sleigh, based on the
friction model used in Section~\ref{sec:sleigh-friction}. We cannot
apply the theory right away since the coordinates $u,v,\omega$ do not
correspond to an orthogonal frame. This can be seen from the
Lagrangian
\begin{equation}\label{eq:lagr-sleigh-frame}
  L = \frac{1}{2}m(u^2 + v^2) + \frac{1}{2}(I+ma^2)\omega^2 + ma \omega v
\end{equation}
with respect to these coordinates: it is purely kinetic, but the
metric is not diagonal with respect to $u,v,\omega$.
We replace $\omega$ by a new coordinate
$\psi = \omega + \frac{ma}{I + ma^2}v$ which diagonalizes the metric.
That is, the coordinates $u,v,\psi$ correspond to an orthogonal frame
and we have
\begin{equation}\label{eq:lagr-sleigh-ortho-frame}
  L = \frac{1}{2}m u^2 +\frac{1}{2}\frac{I m}{I + ma^2} v^2
     +\frac{1}{2}(I+ma^2) \psi^2.
\end{equation}
With respect to these coordinates the equations of
motion~\eqref{eq:sleight-friction-frame} become
\begin{equation}\label{eq:sleight-friction-ortho-frame}
  \begin{split}
    \dot{u} + \frac{ma^2 - I}{I + ma^2} v \psi
            + \frac{m a I}{(I + ma^2)^2} v^2 - a \psi^2 &= 0, \\[3pt]
    \dot{v} + u \psi -\frac{ma}{I + ma^2} u v
                      &= -\frac{I + ma^2}{\varepsilon m I} v, \\[3pt]
    \dot{\psi} + \frac{m a}{I + ma^2} u \psi
               - \frac{m^2 a^2}{(I + ma^2)^2} u v       &= 0.
  \end{split}
\end{equation}
Note that the terms without time derivatives on the left-hand side are
exactly those arising from the connection coefficients, see
Appendix~\ref{sec:conn-form-struc-func}. The friction force only
appears in the equation for $v$ since the frame component associated
to $v$ spans $\Ddist^\perp$, while $F_f$ takes values in $\Ddist^0$
and hence the term $\kappa^\sharp \cdot F_f$,
see~\eqref{eq:EL-friction-coords}, takes values in $\Ddist^\perp$.

To obtain the first order vector field~\eqref{eq:X-first-order}, we
first have to recover $h^{(1)}(x,y,\varphi,u,\psi) = h^{(1)}(u,\psi)$
from~\eqref{eq:h-first-order}. We note that $\xi = (u,0,\psi)$,
$\eta = (0,v,0)$, $V \equiv 0$, that $\omega(f_q\;\cdot\,)$ are
precisely the connection coefficients, and finally that
$(\nu_q^\sharp \, k_q^\flat)_f$ acting on $\eta$ is given by
\begin{equation*}
  (\nu_q^\sharp \, k_q^\flat)_f = 1 \cdot \frac{I m}{I + ma^2}.
\end{equation*}
This yields
\begin{equation*}
  h^{(1)}(u,\psi) = -\frac{I m}{I + ma^2} u \psi
\end{equation*}
since $u\psi$ is the only term in the $v$-component of the connection
one-form in~\eqref{eq:sleight-friction-ortho-frame} that survives when
we insert $\xi$, i.e.~$v = 0$. Then we recover for $\tilde{X}^{(1)}$:
\begin{equation*}
  \begin{alignedat}{2}
    (\dot{x},\dot{y}) &= -\frac{I m}{I + ma^2} u \psi (-\sin(\varphi),\cos(\varphi)), & \qquad
    \dot{\varphi}     &= -\frac{I m^2 a}{(I + ma^2)^2} u \psi, \\
    \dot{u}    &=  \frac{I m (ma^2 - I)}{(I + ma^2)^2} u\,\psi^2, & \qquad
    \dot{\psi} &= -\frac{I m^3 a^2}{(I + ma^2)^3} u^2\,\psi.
  \end{alignedat}
\end{equation*}
Here we used that $\dot{\varphi}=\omega=\psi - \frac{ma}{I + ma^2} v$
and that $\psi = \omega$ on $\Ddist$.

Now we can numerically integrate the vector field
$X_\textup{NH} + \varepsilon \tilde{X}^{(1)}$.
This should, on $\Ddist$, be a first order approximation of the
singularly perturbed vector field $X_\varepsilon$ with friction.
Indeed, the phase and trajectory plots in
Figures~\ref{fig:sleigh-allphaseplots}
and~\ref{fig:sleigh-alltrajplots} clearly show that the green curves
of the first order expansion $X_\textup{NH} + \varepsilon\tilde{X}^{(1)}$
converge more than linearly in $\varepsilon$ to the red curves of the
singularly perturbed system~$X_\varepsilon$.

\begin{figure}[htb]
  \captionsetup{margin=0.1cm}
  \centering
  \begin{minipage}[t]{7.5cm}
    \centering
    \includegraphics[width=7.5cm]{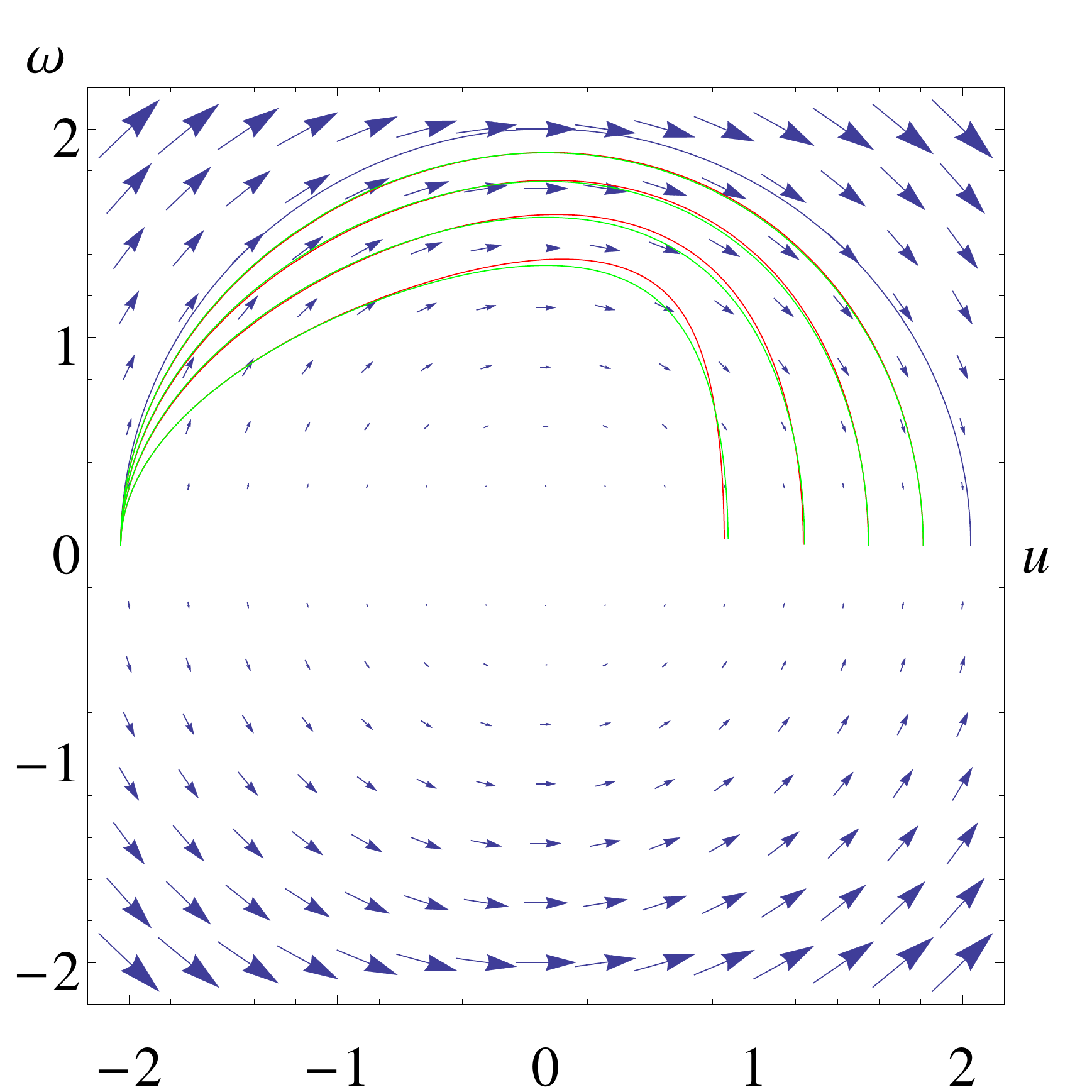}
    \caption{Phase plot of the $(u,\omega)$ coordinates with
      the nonholonomic system in blue, the system with friction in red
      and the first order approximation in green.}
    \label{fig:sleigh-allphaseplots}
  \end{minipage}
  \hspace{0.2cm}
  \begin{minipage}[t]{8cm}
    \centering
    \includegraphics[width=8cm]{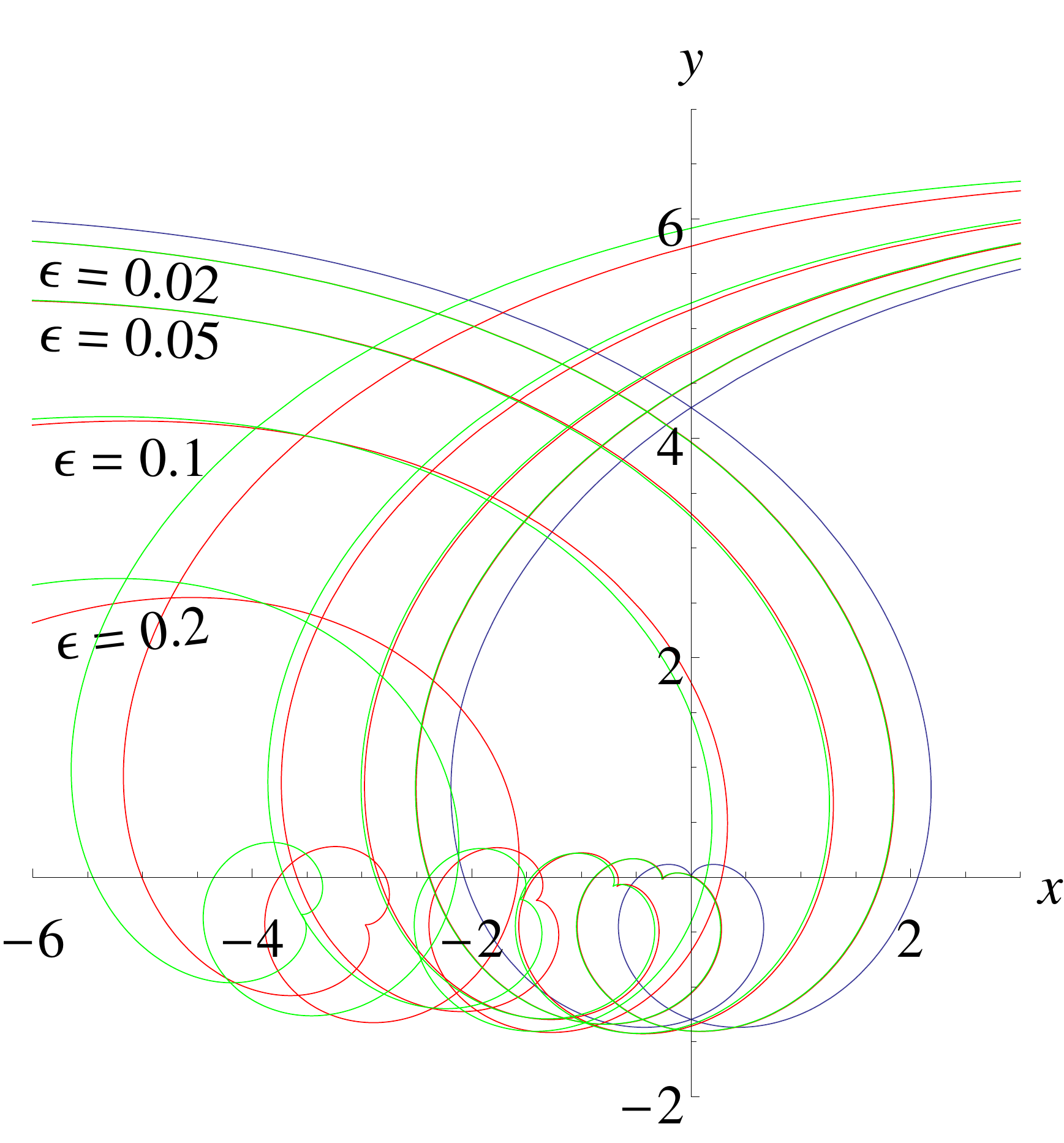}
    \caption{Trajectory plots associated to the orbits in
      Figure~\ref{fig:sleigh-allphaseplots}.}
    \label{fig:sleigh-alltrajplots}
  \end{minipage}
\end{figure}

\section*{Acknowledgments}

This topic was initially suggested to me by Hans Duistermaat as PhD
research project. Although my PhD research finally went another
direction, Hans' insights have been invaluable. Since then, this topic
has been in the back of my mind and during my stay at PUC-Rio I
started actively working on it again and gave a couple of lectures.
This paper grew out of the accompanying lecture notes; I like to thank
Alex Castro for his suggestion to do so. I want to thank Paula
Balseiro and Luis Garc\'ia Naranjo for stimulating discussions and
helpful remarks and Jair Koiller and the people at the mathematics
department at PUC-Rio for their hospitality. This research was
supported by the Capes grant PVE11-2012.

\appendix

\section{Connection form vs.\ structure functions}
\label{sec:conn-form-struc-func}

In this appendix we briefly recall the connection one-form as a method
to express (geodesic) equations of motion with respect to a moving
frame. We also relate this to the formulation using the structure
functions associated to the frame in terms of the Lie brackets of the
vector fields spanning the moving frame. The latter formulation is
more common in nonholonomic dynamics and known as Hamel's formalism,
see e.g.~\cite{Bloch2009:quasivel-sym-nonhol}.

Let $f\colon Q \times \R^n \to \T Q$ be a (local) moving frame, where
$f_\alpha \in \mathfrak{X}(Q)$ denote the individual vector fields
spanning the frame and $f_\alpha = f^i_\alpha\,\partial_i$ their
decomposition with respect to a basis induced by local coordinates
$q^i$. Let $\lambda\colon \T Q \to Q \times \R^n$ denote the inverse
of $f$, that is, a (local) trivialization of $\T Q$. Roman indices are
used for induced coordinates; Greek ones to index moving frame
coordinates. Now a connection $\nabla$ on $Q$ can be expressed with
respect to the frame $f$ as
\begin{equation}\label{eq:conn-form}
  \nabla_{f_\gamma} f_\beta
  = \omega(f_\gamma)^\alpha_\beta \, f_\alpha
  = \omega^\alpha_{\beta\gamma}   \, f_\alpha,
\end{equation}
where $\omega \in \Omega^1(Q;\textrm{End}(\R^n))$ is the connection
one-form and $\omega^\alpha_{\beta\gamma}$ are its coefficients.
On the other hand, the structure functions $C^\alpha_{\beta\gamma}$
encode the Lie brackets relative to a frame as follows:
\begin{equation}\label{eq:struc-func}
  \lie{f_\beta}{f_\gamma} = C^\alpha_{\beta\gamma} \, f_\alpha.
\end{equation}

Let $\nabla$ be the Levi-Civita connection associated to a metric
$\kappa$ on $Q$. According
to~\cite[Prop.~2.3]{Kobayashi1963:diff-geom-vol1} we have
\begin{equation*}
  2\kappa(\nabla_X Y, Z)
  = \begin{aligned}[t]
    &X \cdot \kappa(Y,Z) + Y \cdot \kappa(X,Z) - Z \cdot \kappa(X,Y) \\
    &+\kappa(\lie{X}{Y},Z) + \kappa(\lie{Z}{X},Y) + \kappa(\lie{Z}{Y},X)
  \end{aligned}
\end{equation*}
for any vector fields $X,Y,Z \in \mathfrak{X}(Q)$.
When we decompose the vector fields with respect to the frame $f$,
i.e.\ write $X = X^\alpha f_\alpha$,
this yields the following relation between the connection
coefficients, the metric, and the structure functions:
\begin{equation}\label{eq:rel-conn-form-struc-func}
  \omega^\eta_{\gamma\delta}
  = \frac{1}{2} \kappa^{\eta\beta}\bigl(
      f_\delta(\kappa_{\gamma\beta}) + f_\gamma(\kappa_{\delta\beta})
                                     - f_\beta(\kappa_{\delta\gamma}) \bigr)
   +\frac{1}{2} \bigl(
      C^\eta_{\delta\gamma}
     +\kappa_{\alpha\gamma} \kappa^{\eta\beta} C^\alpha_{\beta\delta}
     -\kappa_{\alpha\delta} \kappa^{\eta\beta} C^\alpha_{\beta\gamma}
    \bigr).
\end{equation}
Note that when $f$ is a holonomic, coordinate-induced frame, then the
structure constants are zero and the first terms reduce to the usual
Christoffel symbols. Conversely, the torsion-freeness of
$\nabla$, i.e.\ $\lie{X}{Y} = \nabla_X Y - \nabla_Y X$ for any
$X,Y \in \mathfrak{X}(Q)$, implies that
\begin{equation*}
  C^\alpha_{\beta\gamma} = \omega^\alpha_{\gamma\beta} - \omega^\alpha_{\beta\gamma}.
\end{equation*}

We return to the formulation of Lagrangian mechanics.
For simplicity we consider a purely kinetic\footnote{%
  A potential term would add a force field that transforms
  covariantly, hence is trivial to add afterwards.%
} Lagrangian $L(q,\dot{q}) = \frac{1}{2}\kappa_q(\dot{q},\dot{q})$, so
its Euler--Lagrange equations correspond to the geodesic equations
$\nabla_{\dot{q}} \dot{q} = 0$, or in induced coordinates,
$\ddot{q}^i = - \Gamma^i_{jk} \dot{q}^j \dot{q}^k$, where $\Gamma$ are
the Christoffel symbols of the metric $\kappa$. Denoting frame
coordinates by $v^\alpha$ (these are also called quasi-velocities),
we have
\begin{equation}\label{eq:EL-conn-form}
  \dot{v}^\alpha = -\omega^\alpha_{\beta\gamma} v^\beta v^\gamma.
\end{equation}
The coefficients $\omega^\alpha_{\beta\gamma}$ play the same role as
the Christoffel symbols $\Gamma^i_{jk}$, but with respect to the
moving frame $f$. They are related by
\begin{equation}\label{eq:form-Christoffel-sym}
  \omega^\alpha_{\beta\gamma}
  = \lambda^\alpha_i \Gamma^i_{jk} f^j_\beta f^k_\gamma
   +\lambda^\alpha_i (\partial_k f^i_\beta)  f^k_\gamma
\end{equation}
since $\Gamma$ represents the connection with respect to the local
coordinate frame and~\eqref{eq:form-Christoffel-sym} expresses the
change to the frame $f$.

On the other hand, the Euler--Lagrange equations with respect to a
moving frame are given
by, see e.g.~\cite[Prop.~1.4.6]{Cushman2010:geom-nonhol-systems},
\cite[Eq.~(2.5)]{Bloch2009:quasivel-sym-nonhol},
\begin{equation*}\label{eq:EL-moving-frame}
  \der{}{t}\pder{\mathcal{L}}{v^\alpha} - \pder{\mathcal{L}}{q^i}\,f^i_\alpha
  - \pder{\mathcal{L}}{v^\gamma} C^\gamma_{\beta\alpha} v^\beta = 0
\end{equation*}
with $\mathcal{L} = L \circ f\colon Q \times \R^n \to \R$ the
Lagrangian with respect to the frame. In our case this boils down to
\begin{align}
0 &= \der{}{t}\pder{\mathcal{L}}{v^\alpha} - \pder{\mathcal{L}}{q^i}\,f^i_\alpha
    -\pder{\mathcal{L}}{v^\gamma} C^\gamma_{\beta\alpha} v^\beta \notag\\
  &= \der{}{t}\Bigl[ \kappa_{\alpha\beta} v^\beta \Bigr]
    -\frac{1}{2} f_\alpha\cdot\kappa_{\beta\gamma}\, v^\beta v^\gamma
    -\kappa_{\delta\gamma} C^\delta_{\beta\alpha} v^\beta v^\gamma \notag\\
  &= \kappa_{\alpha\beta}\dot{v}^\beta + f_\gamma\cdot\kappa_{\alpha\beta}\, v^\beta v^\gamma
    -\frac{1}{2} f_\alpha\cdot\kappa_{\beta\gamma}\, v^\beta v^\gamma
    -\kappa_{\delta\gamma} C^\delta_{\beta\alpha} v^\beta v^\gamma \notag\\
\Leftrightarrow\quad
\label{eq:EL-struc-func}
\dot{v}^\alpha
  &= -\kappa^{\alpha\delta}\bigl[
                   f_\gamma\cdot\kappa_{\delta\beta}
       -\frac{1}{2}f_\delta\cdot\kappa_{\beta\gamma}
       +\kappa_{\eta\gamma} C^\eta_{\beta\alpha} \bigr] v^\beta v^\gamma.
\end{align}
Note that by equating~\eqref{eq:EL-conn-form}
and~\eqref{eq:EL-struc-func} we also find an explicit relation between
the connection one-form $\omega$, the metric $\kappa$ and the
structure functions $C$. It differs
from~\eqref{eq:rel-conn-form-struc-func} only by terms that are
anti-symmetric in the two lower indices, since these cancel in the
geodesic equation.

\bibliographystyle{amsalpha}
\bibliography{bibfile}

\newcommand{\etalchar}[1]{$^{#1}$}
\def\polhk#1{\setbox0=\hbox{#1}{\ooalign{\hidewidth
  \lower1.5ex\hbox{`}\hidewidth\crcr\unhbox0}}} \def\cprime{$'$} \def\ui{i}
\providecommand{\bysame}{\leavevmode\hbox to3em{\hrulefill}\thinspace}
\providecommand{\MR}{\relax\ifhmode\unskip\space\fi MR }
\providecommand{\MRhref}[2]{%
  \href{http://www.ams.org/mathscinet-getitem?mr=#1}{#2}
}
\providecommand{\href}[2]{#2}
\begin{thebibliography}{BKM{\etalchar{+}}15}

\bibitem[App11]{Appell1911:ex-mouv-rel-nonlin-vitesse}
M.~Paul Appell, \emph{Exemple de mouvement d’un point assujetti é une
  liaison exprimée par une relation non linéaire entre les composantes de la
  vitesse}, Rend. Circ. Mat. Palermo \textbf{32} (1911), no.~1, 48--50.

\bibitem[Arn78]{Arnold1978:ODEs}
V.~I. Arnol{\cprime}d, \emph{Ordinary differential equations}, MIT Press,
  Cambridge, Mass.-London, 1978, Translated from the Russian and edited by
  Richard A. Silverman. \MR{0508209 (58 \#22707)}

\bibitem[BKM{\etalchar{+}}15]{Borisov2015:exp-motion-sliding}
Alexey~V. Borisov, Yury~L. Karavaev, Ivan~S. Mamaev, Nadezhda~N. Erdakova,
  Tatyana~B. Ivanova, and Valery~V. Tarasov, \emph{Experimental investigation
  of the motion of a body with an axisymmetric base sliding on a rough plane},
  Regul. Chaotic Dyn. \textbf{20} (2015), no.~5, 518--541.

\bibitem[Blo03]{Bloch2003:nonholmech}
A.~M. Bloch, \emph{Nonholonomic mechanics and control}, Interdisciplinary
  Applied Mathematics, vol.~24, Springer-Verlag, New York, 2003, With the
  collaboration of J. Baillieul, P. Crouch and J. Marsden, With scientific
  input from P. S. Krishnaprasad, R. M. Murray and D. Zenkov, Systems and
  Control. \MR{MR1978379 (2004e:37099)}

\bibitem[Blo10]{Bloch2010:nonhol-dissip-quant}
Anthony~M. Bloch, \emph{Nonholonomic mechanics, dissipation and quantization},
  Advances in the theory of control, signals and systems with physical
  modeling, Lecture Notes in Control and Inform. Sci., vol. 407, Springer,
  Berlin, 2010, pp.~141--152. \MR{2765956 (2012c:70020)}

\bibitem[BMZ09]{Bloch2009:quasivel-sym-nonhol}
Anthony~M. Bloch, Jerrold~E. Marsden, and Dmitry~V. Zenkov,
  \emph{Quasivelocities and symmetries in non-holonomic systems}, Dyn. Syst.
  \textbf{24} (2009), no.~2, 187--222. \MR{2542960 (2011b:70020)}

\bibitem[BR08]{Bloch2008:quant-nonhol}
Anthony~M. Bloch and Alberto~G. Rojo, \emph{Quantization of a nonholonomic
  system}, Phys. Rev. Lett. \textbf{101} (2008), no.~3, 030402, 4. \MR{2430224
  (2009j:81071)}

\bibitem[Bre81]{Brendelev1981:realnonhol}
V.~N. Brendelev, \emph{On the realization of constraints in nonholonomic
  mechanics}, J. Appl. Math. Mech. \textbf{45} (1981), no.~3, 481--487.
  \MR{MR661547 (83k:70018)}

\bibitem[BRMR08]{Bou-Rabee2004:dissipation-tippetop}
Nawaf~M. Bou-Rabee, Jerrold~E. Marsden, and Louis~A. Romero,
  \emph{Dissipation-induced heteroclinic orbits in tippe tops}, SIAM Rev.
  \textbf{50} (2008), no.~2, 325--344, Revised reprint of SIAM J. Appl. Dyn.
  Syst. {{\bf{3}}} (2004), no. 3, 352--357 [MR2114737]. \MR{2403054}

\bibitem[Car33]{Caratheodory1933:schlitten}
C.~Carath{\'e}odory, \emph{Der schlitten}, Z. Angew. Math. Mech. \textbf{13}
  (1933), no.~2, 71--76.

\bibitem[CD{\'S}10]{Cushman2010:geom-nonhol-systems}
Richard Cushman, Hans Duistermaat, and J{\polhk{e}}drzej {\'S}niatycki,
  \emph{Geometry of nonholonomically constrained systems}, Advanced Series in
  Nonlinear Dynamics, vol.~26, World Scientific Publishing Co. Pte. Ltd.,
  Hackensack, NJ, 2010. \MR{2590472 (2011f:37113)}

\bibitem[Cha11]{Chaplygin1911:motion-nonhol}
S.~A. Chaplygin, \emph{On the theory of motion of nonholonomic systems. {T}he
  reducing-multiplier theorem}, Mat. sb. \textbf{28} (1911), no.~1, 303--314,
  In Russian.

\bibitem[Cha08]{Chaplygin1911:motion-nonhol-english2008}
\bysame, \emph{On the theory of motion of nonholonomic systems. {T}he
  reducing-multiplier theorem}, Regul. Chaotic Dyn. \textbf{13} (2008), no.~4,
  369--376, Translated from {{\i}t Matematicheski{\u\i} sbornik} (Russian)
  {{\bf{2}}8} (1911), no. 1 by A. V. Getling. \MR{2456929 (2010a:70017)}

\bibitem[dL12]{Leon2012:hist-nonhol-mech}
Manuel de~Le{\'o}n, \emph{A historical review on nonholomic mechanics}, Rev. R.
  Acad. Cienc. Exactas F\'\i s. Nat. Ser. A Math. RACSAM \textbf{106} (2012),
  no.~1, 191--224. \MR{2892144}

\bibitem[Eld13]{Eldering2013:NHIM-noncompact}
Jaap Eldering, \emph{Normally hyperbolic invariant manifolds --- the noncompact
  case}, Atlantis Series in Dynamical Systems, vol.~2, Atlantis Press, Paris,
  September 2013. \MR{3098498}

\bibitem[Fen72]{Fenichel1971:invarmflds}
Neil Fenichel, \emph{Persistence and smoothness of invariant manifolds for
  flows}, Indiana Univ. Math. J. \textbf{21} (1971/1972), 193--226.
  \MR{MR0287106 (44 \#4313)}

\bibitem[Fen79]{Fenichel1979:singpertODE}
\bysame, \emph{Geometric singular perturbation theory for ordinary differential
  equations}, J. Differential Equations \textbf{31} (1979), no.~1, 53--98.
  \MR{MR524817 (80m:58032)}

\bibitem[FGN10]{Fedorov2010:hydro-chaplygin-sleigh}
Yuri~N. Fedorov and Luis~C. Garc{\'{\i}}a-Naranjo, \emph{The hydrodynamic
  {C}haplygin sleigh}, J. Phys. A \textbf{43} (2010), no.~43, 434013, 18.
  \MR{2727787 (2011j:74042)}

\bibitem[Fuf64]{Fufaev1964:real-nonhol-visc-fric}
N.A. Fufaev, \emph{On the possibility of realizing a nonholonomic constraint by
  means of viscous friction forces}, J. Appl. Math. Mech. \textbf{28} (1964),
  no.~3, 630--632.

\bibitem[HPS77]{Hirsch1977:invarmflds}
M.~W. Hirsch, C.~C. Pugh, and M.~Shub, \emph{Invariant manifolds}, Lecture
  Notes in Mathematics, vol. 583, Springer-Verlag, Berlin, 1977. \MR{MR0501173
  (58 \#18595)}

\bibitem[Kar81]{Karapetian1981:realnonhol}
A.~V. Karapetian, \emph{On realizing nonholonomic constraints by viscous
  friction forces and {C}eltic stones stability}, J. Appl. Math. Mech.
  \textbf{45} (1981), no.~1, 42--51. \MR{MR654774 (83f:70013)}

\bibitem[KN63]{Kobayashi1963:diff-geom-vol1}
Shoshichi Kobayashi and Katsumi Nomizu, \emph{Foundations of differential
  geometry. {V}ol {I}}, Interscience Publishers, a division of John Wiley \&
  Sons, New York-London, 1963. \MR{0152974 (27 \#2945)}

\bibitem[KN90]{Kozlov1990:real-hol-constr}
V.~V. Kozlov and A.~I. Ne{\u\i}shtadt, \emph{Realization of holonomic
  constraints}, Prikl. Mat. Mekh. \textbf{54} (1990), no.~5, 858--861.
  \MR{1088212 (92b:70014)}

\bibitem[Koz82a]{Kozlov1982:dyn-nonint-constr1}
V.~V. Kozlov, \emph{The dynamics of systems with nonintegrable constraints.
  {I}}, Vestnik Moskov. Univ. Ser. I Mat. Mekh. (1982), no.~3, 92--100, 112.
  \MR{671067 (84i:70024a)}

\bibitem[Koz82b]{Kozlov1982:dyn-nonint-constr2}
\bysame, \emph{The dynamics of systems with nonintegrable constraints. {II}},
  Vestnik Moskov. Univ. Ser. I Mat. Mekh. (1982), no.~4, 70--76, 86. \MR{671890
  (84i:70024b)}

\bibitem[Koz83]{Kozlov1983:real-nonint-constr}
\bysame, \emph{Realization of nonintegrable constraints in classical
  mechanics}, Dokl. Akad. Nauk SSSR \textbf{272} (1983), no.~3, 550--554.
  \MR{723778 (84m:70024)}

\bibitem[Koz92]{Kozlov1992:realconstr}
\bysame, \emph{On the realization of constraints in dynamics}, Prikl. Mat.
  Mekh. \textbf{56} (1992), no.~4, 692--698. \MR{MR1191861 (93m:70015)}

\bibitem[Koz10]{Kozlov2010:dry-fric-nonhol}
\bysame, \emph{Note on dry friction and non-holonomic constraints}, Nelim.
  Dinam. \textbf{6} (2010), no.~4, 903--906, In Russian.

\bibitem[LM95]{Lewis1995:var-constrsys-th-exp}
Andrew~D. Lewis and Richard~M. Murray, \emph{Variational principles for
  constrained systems: theory and experiment}, Internat. J. Non-Linear Mech.
  \textbf{30} (1995), no.~6, 793--815. \MR{MR1365861 (96j:70017)}

\bibitem[Mar98]{Marle1998:approaches-nonhol}
Charles-Michel Marle, \emph{Various approaches to conservative and
  nonconservative nonholonomic systems}, Rep. Math. Phys. \textbf{42} (1998),
  no.~1-2, 211--229, Pacific Institute of Mathematical Sciences Workshop on
  Nonholonomic Constraints in Dynamics (Calgary, AB, 1997). \MR{1656282
  (2000a:70021)}

\bibitem[NF72]{Neimark1972:dynnonhol}
Yu.~I. Ne{\u\i}mark and N.~A. Fufaev, \emph{Dynamics of nonholonomic systems},
  Translations of mathematical monographs, vol.~33, American Mathematical
  Society, Providence, RI, 1972.

\bibitem[RU57]{Rubin1957:motion-constr}
Hanan Rubin and Peter Ungar, \emph{Motion under a strong constraining force},
  Comm. Pure Appl. Math. \textbf{10} (1957), 65--87. \MR{0088162 (19,477c)}

\bibitem[Rui98]{Ruina1998:nonhol-stab-piecewise}
Andy Ruina, \emph{Nonholonomic stability aspects of piecewise holonomic
  systems}, Rep. Math. Phys. \textbf{42} (1998), no.~1-2, 91--100, Pacific
  Institute of Mathematical Sciences Workshop on Nonholonomic Constraints in
  Dynamics (Calgary, AB, 1997). \MR{1656277 (99k:70017)}

\bibitem[SS08]{Sidek2008:model-nonhol-robot-slip}
N.~Sidek and N.~Sarkar, \emph{Dynamic modeling and control of nonholonomic
  mobile robot with lateral slip}, ICONS '08. Third International Conference on
  Systems, April 2008, pp.~35--40.

\bibitem[Tak80]{Takens1980:motion-constr}
Floris Takens, \emph{Motion under the influence of a strong constraining
  force}, Global theory of dynamical systems ({P}roc. {I}nternat. {C}onf.,
  {N}orthwestern {U}niv., {E}vanston, {I}ll., 1979), Lecture Notes in Math.,
  vol. 819, Springer, Berlin, 1980, pp.~425--445. \MR{591202 (82g:34060)}

\bibitem[Tih52]{Tikhonov1952:sys-DE-small-param}
A.~N. Tihonov, \emph{Systems of differential equations containing small
  parameters in the derivatives}, Mat. Sbornik N. S. \textbf{31(73)} (1952),
  575--586. \MR{0055515 (14,1085d)}

\bibitem[VG88]{Vershik1988:nonhol-theory-distrib}
A.~M. Vershik and V.~Ya. Gershkovich, \emph{Nonholonomic problems and the
  theory of distributions}, Acta Appl. Math. \textbf{12} (1988), no.~2,
  181--209. \MR{966452 (90d:58010)}

\bibitem[WH96]{Wang1996:creep-dyn-nonhol}
Jiunn-Cherng Wang and Han-Pang Huang, \emph{Creep dynamics of nonholonomic
  systems}, Proceedings of the 1996 IEEE International Conference on Robotics
  and Automation, vol.~4, 1996, pp.~3452--3457.

\end{thebibliography}

\end{document}